\documentclass[a4paper,12pt]{article}

\usepackage[utf8]{inputenc}
\usepackage[T1]{fontenc}
\usepackage[english]{babel}

\usepackage{enumerate}
\usepackage{amsmath,amsfonts,amssymb,amsthm}
\usepackage{graphicx,tabularx,color,subfigure}
\usepackage{enumitem}
\usepackage[all]{xy}

\newtheorem{defn}{Definition}
\newtheorem{lem}{Lemma}

\newtheorem{pro}{Proposition}
\newtheorem{rem}{Remark}
\newtheorem{cor}{Corollary}

\newcommand{\C}{\mathbb{C}}
\newcommand{\CC}{\widehat{\mathbb{C}}}
\newcommand{\D}{\mathbb{D}}

\newcommand{\R}{\mathbb{R}}

\newcommand{\M}{\mathcal{M}}
\newcommand{\vf}{\frac{d}{dz}}

\newtheorem*{thmA}{Theorem A}

\renewcommand{\b}{\color{blue}}

\title{\bf On the separatrix graph of a rational vector field on the Riemann sphere}

\author{ Kealey Dias\thanks{The first author has been partially supported by the Association for Women in Mathematics  Travel Grant  October 2019 Cycle (NSF 1642548) and the second author has been partially supported by MINECO-AEI grant
MTM-2017-86795-C3-2-P.}\\
{\small Bronx Community College of the City University of New York}\\
{\small 2155 University Ave. Bronx, New York, USA 10453. }\\
{\small kealey.dias@gmail.com} \and
 Antonio Garijo\\
{\small Dept. d'Enginyeria Inform\`atica i Matem\`atiques. Universitat Rovira i Virgili}\\
{\small Av. Pa\"isos Catalans 26. Tarragona 43007, Spain.}\\
{\small antonio.garijo@urv.cat}  }
\begin{document}

\maketitle

\begin{abstract}
{ We consider  the rational flow $\xi_R(z)= R(z) (d/dz)$ where $R$ is given by the quotient of two polynomials without common factors  on the Riemann sphere. The separatrix graph $\Gamma_R$  is the boundary between trajectories with different properties. We characterize the properties of a planar directed graph to be the separatrix graph of a rational vector field on the Riemann sphere.}
\vspace{0.5cm}


\vglue 0.2truecm

\end{abstract}

\section{Introduction}\label{sec:intro}

Complex differential equations have been playing an important role both in theoretical and  in applied mathematics. Therefore, the study of these continuous dynamical systems could be interesting for wide areas of knowledge, from complex geometry to fluid dynamics.  In this work we investigate some properties of a rational vector field defined in the Riemann sphere $\CC$. More precisely, we consider

\begin{equation}\label{eq:Main}
\frac{dz}{dt} = R(z) , \qquad t \in \R
\end{equation}
\noindent defined by a rational map $R: \CC \to \CC$ given by $R(z)=P(z)/Q(z)$ where $P$ and $Q$ are two polynomials that we assume without common factors. We also denote by $n$ and $m$ the degree of the
polynomials $P$ and $Q$, respectively. We assume that the degree of $R$, denoted by $deg{(R)} = \max{\{n,m\}}$, is bigger or equal than $2$.   Equivalently,   $\xi_R(z)= R(z) (d/dz)$ denotes the rational vector field associated  to  $R$. We concentrate in rational vector fields on the Riemann sphere. For a deep investigation on other Riemann surfaces  we refer to \cite{FlowsRiemannSurfaces} and in the presence of an essential singularity to  \cite{EssentialSingularity}.

These type of vector fields has been widely investigated  from the local and global point of view (see  for instance \cite{NotesMeromorphicPartI,NotesMeromorphicPartII, ConformalEquivalence, VectorFieldsComplexFunctions,MeromorphicComplexDifferentialEquations, PlaneAutonomous,ComplexStructures,LocalGlobalPhasePortrait}) and we quote briefly some well  known results of these systems.  The fact that the vector field $\xi_R$ comes from a meromorphic map $R$ provides specific properties to these kind of continuous dynamical systems. Among others $\xi_R$  does not present limit cycles or there are not center-focus problem  since the linear part of a critical point determines its behavior. More precisely, any simple zero $\alpha$ of  $P$ gives  a  {\it sink/source} or a {\it center}  of $\xi_R$ according to $Re(P'(\alpha))$  is different or equal to zero, respectively. A multiple zero of $P$  gives a critical point of $\xi_R$ with elliptic sectors. Finally, any root of $Q$ gives a  generalized saddle of the rational flow with  hyperbolic sectors. A phase curve is called a  {\it separatrix}  if it tends to a saddle point.  The {\it separatrix graph}, denoted by $\Gamma_R$, of the rational flow $\xi_R$ is formed by the union of the closure of  all the sepatrices of the system.  The main goal of this work is to find a characterization of the separatrix graph of a rational vector field.

 We mention two particular cases of rational vector fields. The first one is when the rational map is a polynomial $P$.  The critical points  of the polynomial flow $\dot{z}=P(z)$ are the zeros of $P$, and $\infty$ is  a  generalized saddle point. For more details of the classification of polynomial flows via a complete set of realizable topological and analytic invariants, we refer the reader to \cite{ClassificationComplexPolynomial,EnumeratingCombinatorialClasses}.  Our approach will be based on the polynomial case,  taking into account that in the rational case has a richer structure in the phase space.  One example of this difference is that the separatrix graph of polynomial flow is always connected, which is not always  the case for rational flows.

 The second particular case of rational flows we highlight is the {\it Newton's} flow when $$R(z)=-P(z)/P'(z)$$ where $P$ is a polynomial.  In this  case, the  critical points of Newton's flow are sinks (zeros of $P$) and simple saddles (zeros of $P'$) and $\infty$ is a source.  Among other properties the  separatrix graph of a Newton's flow is connected.  In 1985 S. Smale proposed a question related to the characterization of the separatrix graph of a  Newton's flow (\cite{EfficiencyAlgorithms}).   In 1988  M. Shub, D. Tischler and R. Williams solved completely this question (\cite{ NewtonianGraphComplexPolynomial}) characterizing the properties of a planar graph to be the separatrix graph of a Newton's flow.

  Based on the same question, our goal is to find a characterization of the separatrix graph of any rational vector field. An {\it admissible graph} is a planar graph with the main properties of the  separatrix graph of a rational vector field (see  \S \ref{sec:admissible_graph} for a precise definition).  The main objective of this paper is to prove the following result.

\begin{thmA}
  A planar directed graph $\Gamma$ homeomorphic  to the separatrix graph of a  rational vector field  if and only if  $\Gamma$ is an admissible graph.
\end{thmA}

Significant progress has been made towards solving this problem in the work of J. Muci\~no (see \cite{ComplexStructures}). That paper completely characterized the rational flows on compact Riemann surfaces of genus $g\geq 1$. The paper also gave a partial characterization of rational flows on $\CC$: those where the poles and equilibrium points are simple.  It remains to show the complete case of rational flows on $\CC$, which is what we do in this paper.

In  \S \ref{sec:preli} we present the main properties of a rational flow, we define the separatrix graph of a rational vector field and deduce some of its  properties. In  \S \ref{sec:admissible_graph} we introduce the abstract definition of an admissible graph,  based on the properties of the separatrix graph of a rational flow. Finally, in  \S \ref{sec:admissible_separatrix}
we prove the characterization of a separatrix graph of a rational flow stated in Theorem A.

\section{Preliminaries}\label{sec:preli}

In this section we recall the main aspects of the phase portrait of a rational flow (see \cite{NotesMeromorphicPartI,NotesMeromorphicPartII, ConformalEquivalence, VectorFieldsComplexFunctions,MeromorphicComplexDifferentialEquations, PlaneAutonomous,ComplexStructures,LocalGlobalPhasePortrait,ClassificationComplexPolynomial,EnumeratingCombinatorialClasses}). We start studying  the local dynamics of equation (\ref{eq:Main}) in some punctured neighborhood of a critical point. The local behavior of a rational flow around a critical point $z_0$ (see Proposition \ref{pro:index})  comes from its local normal form. For a simple zero of $R$ the dynamics near $z_0$ is conformally conjugate to $\dot{z}=R'(z_0) z$ near the origin; for a  multiple zero of $R$ of order $k \geq 2$  one normal form is given by $\dot{z}= z^k/(c+cz^{k-1})$ near the origin where $c={\rm Res}(1/R,z_0)$. Finally, the rational flow near a pole of $R$ of order $k \geq 1$ is conformally conjugate to $\dot{z}=1/z^k$ near the origin.

 In the next proposition we describe the local phase portrait, and we show the different types of behaviors in Figure  \ref{fig:focals_simple}.

\begin{pro}\label{pro:index} Let $R(z)$ be a holomorphic function in a punctured
neigbourhood of $z_0.$
\begin{enumerate}[label=(\alph*)]
 \item If $z_0$ is a zero of $R$ and $R^{\prime}(z_0)\ne 0$, according to $Re(R'(z_0))<0$,
  $Re(R'(z_0))>0$ or $Re(R'(z_0))=0$, then the phase portrait of \eqref{eq:Main} in
  a neigbourhood of $z_0$ is a {\it sink}, a {\it source} or
an {\it isochronous center}, respectively. In all the cases the
index of $\dot{z}=R(z)$ at $z_0$ is 1.
 \item If $z_0$ is a zero of $R$ of order $k\geq 2$, then the phase portrait of \eqref{eq:Main} in
a neigbourhood of $z_0$  contains the union of $2(k-1)$ elliptic sectors,  so the index of $\dot{z}=R(z)$ at $z_0$ is $k.$
\item If $z_0$ is a pole of $R$ of order $k\geq 1$, the phase portrait of \eqref{eq:Main} in
a neigbourhood of $z_0$ is a union of $2(k+1)$ hyperbolic sectors,
and so the index of $\dot{z}=R(z)$ at $z_0$ is $-k.$
\end{enumerate}
\end{pro}

In Figure  \ref{fig:focals_simple} we sketch the phase portrait of a rational flow in a neighborhood of critical point. In the first row we show the local phase portrait of the simple zeros of $R$ which could be  sinks, sources, or  isochronous  centers. In the second row we show the phase portrait of  a multiple
zero of $R$ with multiplicities $k=2,3$. Finally, in the third row we show the local phase portrait of a pole of $R$ with multiplicities $k=1,2$.

\begin{figure}[ht]
    \centering
     \includegraphics[width=0.75\textwidth]{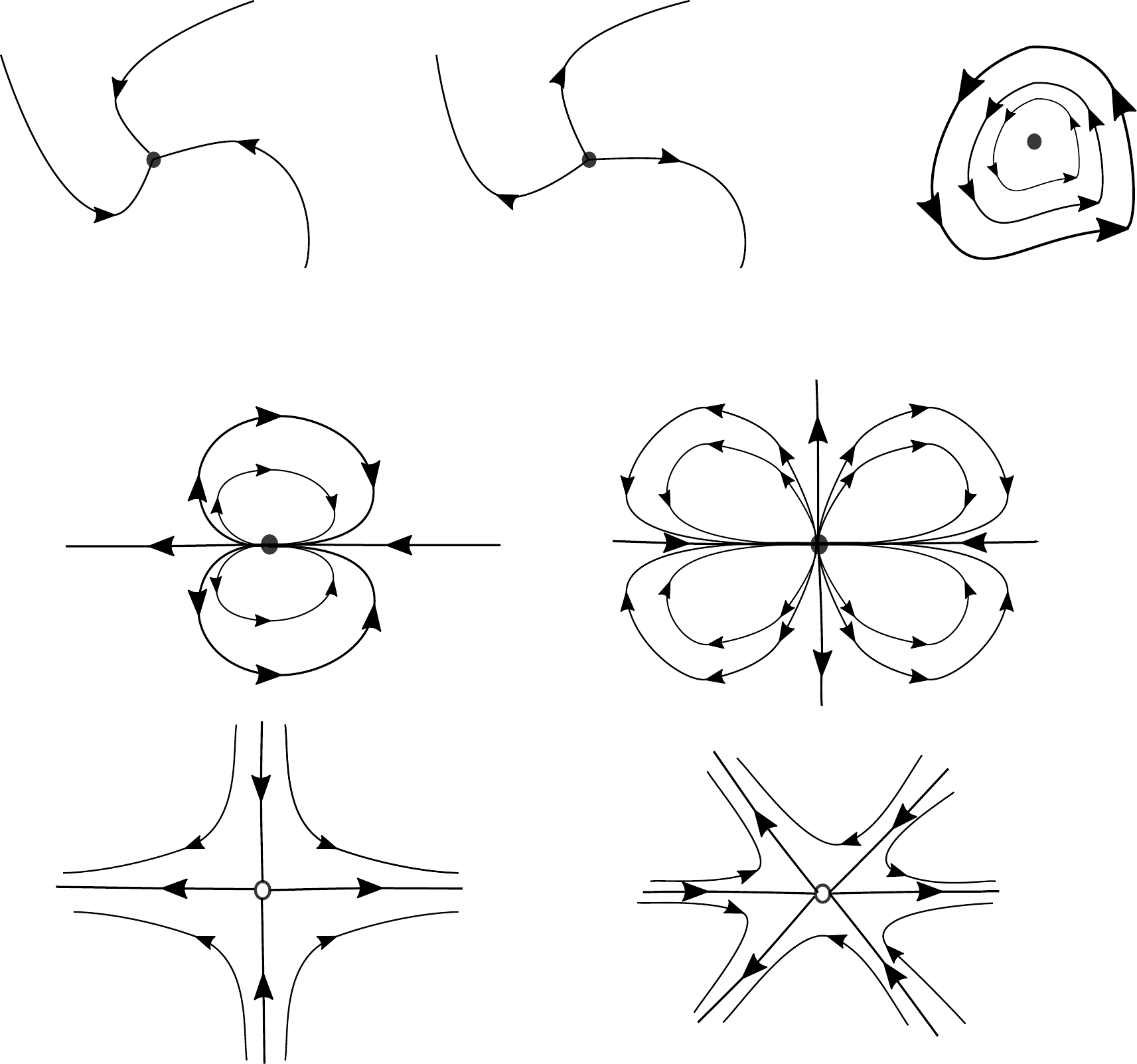}
    \put(-160,228) {\small source }
        \put(-280,228) {\small sink }
    \put(-20,228) {\small center }
    \put(-35,118) {\small elliptic  }
    \put(-35,108) {\footnotesize $(k=3)$ }
   \put(-220,118) {\small elliptic  }
   \put(-220,108) {\footnotesize $(k=2)$  }
      \put(-220,18) {\small saddle }
      \put(-220,8) {\footnotesize $(k=1)$}
     \put(-35,18) {\small saddle }
      \put(-35,8) {\footnotesize $(k=2)$}
           \caption{\small{Different local phase portraits near a critical point in a rational flow. In the first row a sink, a source and a center; in the second row two examples of a critical point of elliptic type; and finally in the third row two  saddles.}}
    \label{fig:focals_simple}
\end{figure}
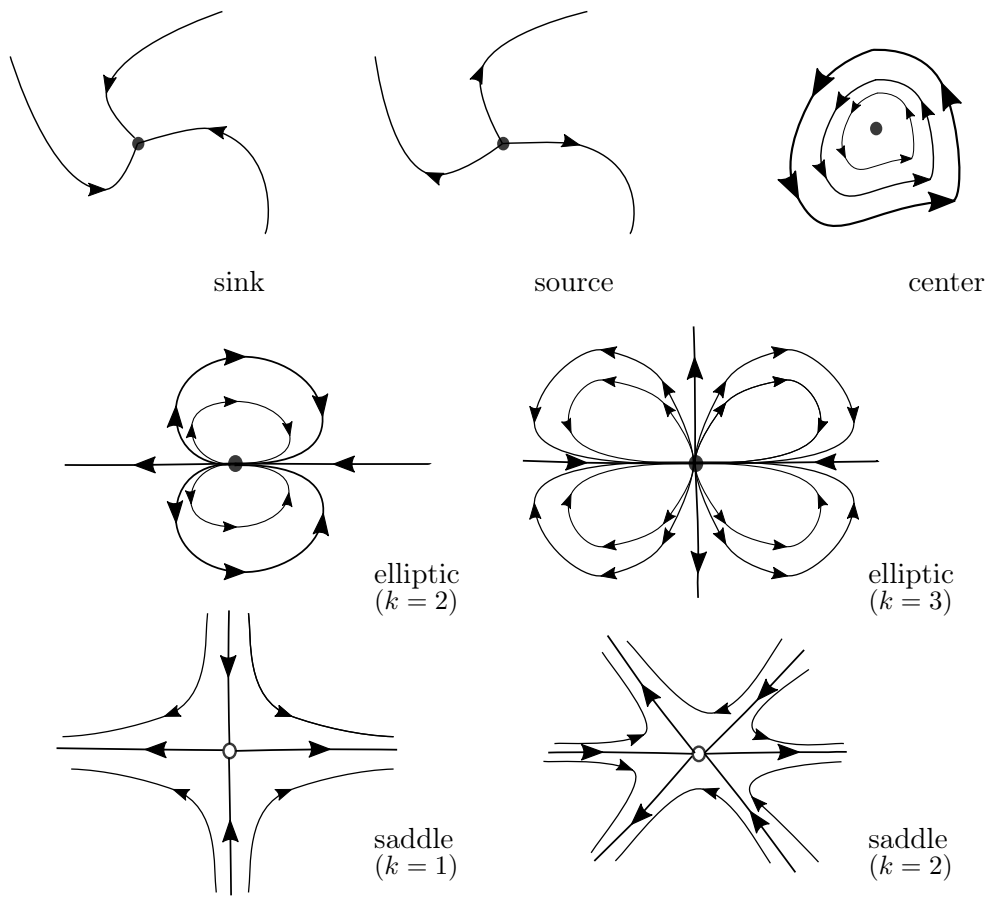

 From Proposition \ref{pro:index}  a  multiple root of $R$ is also called an {\it  elliptic} critical point while a pole of  $R$ is  called  a  {\it  saddle} or a  {\it hyperbolic}  critical point.

    Given any non-critical point $z_0$ in the Riemann sphere, we denote by $\gamma(t,z_0)$ the solution of the Cauchy  problem associated to \eqref{eq:Main} with initial condition $\gamma(0,z_0)=z_0$. Furthermore  we can assume that $\gamma$ is defined in its maximal interval of definition, thus  $\gamma(\cdot,z_0): (t_{min},t_{max}) \to \CC $ is called a {\it trajectory} through $z_0$.  There are only two possibilities for the limit set of a trajectory:  a critical point or a periodic orbit. In the first case critical points are sinks, sources or elliptic points  (corresponding to $t_{min}=-\infty$ and/or $t_{max}=+\infty$) or  saddles (corresponding to   $t_{min}>-\infty$ and/or $t_{max}<+\infty$); and the second case  corresponds to a periodic trajectory $\gamma(t + \tau, z_0) = \gamma(t,z_0)$  of minimal period $\tau>0$. The   Poincar\'e-Bendixon  Theorem of a rational flow says that the  $\omega-$limit set
of an orbit  is either a critical point or a periodic orbit, and  in the second case the  nearby orbit must be periodic. The missing cases (limit cycles and a union of singular points and phase curves) are not allowed due to the holomorphic structure of the flow.

We say that a  trajectory $\gamma$ is a {\it separatrix}  if  it lands in a (generalized) saddle point, or in other words, a separatrix is a trajectory for which the maximal interval of definition is different from $\mathbb R$. There are also different types of separatrices:  an {\it outgoing separatrix} when  $t_{min}>-\infty$ and $t_{max}=\infty$,  an {\it ingoing separatrix} when $t_{min}=-\infty$ and $t_{max}< \infty$,  an {\it heteroclinic separatrix} when $t_{min} > - \infty$, $t_{max}< \infty$ and $$\displaystyle \lim_{t \to t_{min}} \gamma(t,z_0) \neq       \lim_{t \to t_{max}} \gamma(t,z_0),$$
\noindent and finally  an  {\it homoclinic separatrix} when
$t_{min} > - \infty$, $t_{max}< \infty$ and $$\displaystyle \lim_{t \to t_{min}} \gamma(t,z_0) =       \lim_{t \to t_{max}} \gamma(t,z_0).$$
Equivalently,  we could define the separatrices  as the union of all the stable and unstable manifolds of all the hyperbolic points.

There are several equivalent definitions of the separatrix graph of the vector field $\xi_R$ defined on the Riemann sphere $\CC$. We define the separatrix graph as the union of the closure of  all the separatrices, thus
\begin{equation}\label{def:separatrix_graph}
\Gamma_R =  \bigcup_{\gamma \,  separatrix}   \overline{\gamma}.
\end{equation}

We notice that every separatrix $\gamma(\cdot,z_0): (t_{min},t_{max}) \to \CC$ is defined on an open interval, thus we denote by $\overline{\gamma}$ the separatrix $\gamma$ and the two extremities $\displaystyle \lim_{t \to t_{min}} \gamma(t,z_0)$ and $\displaystyle \lim_{t \to t_{max}} \gamma(t,z_0)$. If $\gamma$ is an outgoing/ingoing separatrix one the two extremities is a zero of $R$ and the other is a pole of $R$, moreover when the separatrix is homoclinic/heteroclinic then both extremities are poles of $R$.

\begin{rem}\label{rem:degree2}
If  $R$  is a polynomial of degree 2  using our definition of the separatrix graph \eqref{def:separatrix_graph} we obtain  that $\Gamma_R = \emptyset$ since  there are not any separatrix trajectory.  Moreover, the phase portrait of $R$ is well understood depending on the roots of $R$. Thus,  $\dot{z}=R(z)$ is conformally conjugate to $\dot{z}= i (z-1)(z-1)$ where  $\hat{ \mathbb C} \setminus \{\pm 1\}$ is foliated by isochronous periodic orbits, or to $\dot{z}= (z+1)(z-1)$ where $\hat{ \mathbb C} \setminus \{\pm 1\}$ is foliated by orbits going from -1 to +1, or to $\dot{z} = z^2$ where $\hat{ \mathbb C} \setminus \{0\}$ is foliated by orbits going from 0 to 0.
\end{rem}

From the above remark and hereafter we will  assume that the degree of $R$ is bigger or equal than three and thus $\Gamma_R \neq \emptyset$. The separatrix graph $\Gamma_R$ is the boundary between trajectories with different properties. According the Markus Theorem (\cite{DifferentialEquationsDynamicalSystems}, Thm. {2, Sec. 3.11) the separatrix graph
$\Gamma_R$ is closed, so their complement $\CC \setminus \Gamma_R$ is open. Following   \cite{ClassificationComplexPolynomial,VecteursPolynomiaux}   the connected components of  $\CC \setminus \Gamma_R$ are called  {\it zones} or {\it canonical regions}. Moreover, in every  zone  all the trajectories are equivalent and could be classified into four types. See Figure \ref{fig:zones}.

\begin{figure}[htb!]
    \centering
     \includegraphics[width=0.71\textwidth]{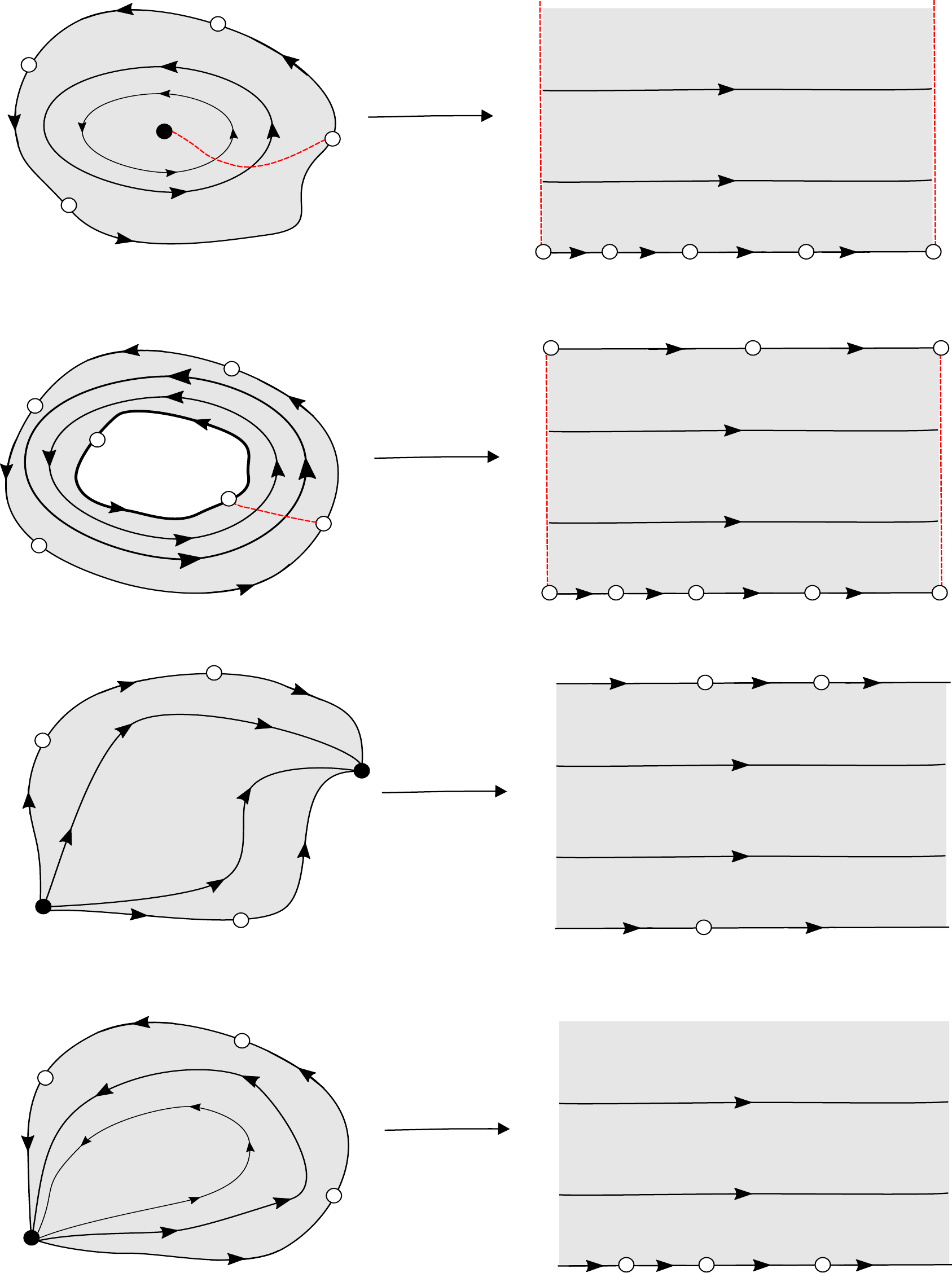}
   \put(-282,394) {\small $z_0$}
          \put(-217,396) {\small $w_1$}
          \put(-192,416) {\small $\phi$}
          \put(-253,441) {\small $w_2$}
          \put(-341,433) {\small $w_3$}
          \put(-321,370) {\small $w_4$}
          \put(-127,350) {\small $w_2$}
         \put(-152,350) {\small $w_1$}
          \put(-97,350) {\small $w_3$}
          \put(-57,350) {\small $w_4$}
         \put(-12,350) {\small $w_1$}
         \put(-82,400) {\small $\mathcal{C}$}
          \put(-387,360) {\small center zone}
          \put(-112,430) {\small half infinite cylinder}
          \put(-217,260) {\small $w_1$}
                \put(-190,295) {\small $\phi$}
          \put(-247,319) {\small $w_2$}
          \put(-336,315) {\small $w_3$}
          \put(-331,246) {\small $w_4$}
          \put(-297,292) {\small $w_6$}
          \put(-267,279) {\small $w_5$}
          \put(-150,228) {\small $w_1$}
         \put(-125,228) {\small $w_2$}
          \put(-95,228) {\small $w_3$}
          \put(-55,228) {\small $w_4$}
         \put(-10,228) {\small $w_1$}
          \put(-150,334) {\small $w_5$}
          \put(-78,334) {\small $w_6$}
          \put(-10,334) {\small $w_5$}
         \put(-80,275) {\small $\mathcal{A}$}
          \put(-387,230) {\small annular zone}
        \put(-100,305) {\small finite cylinder}
          \put(-213,165) {\small $z_1$}
          \put(-185,175) {\small $\phi$}
          \put(-249,117) {\small $w_1$}
          \put(-90,112) {\small $w_1$}
          \put(-325,117) {\small $z_0$}
           \put(-334,193) {\small $w_2$}
           \put(-261,219) {\small $w_3$}
          \put(-90,215) {\small $w_2$}
          \put(-50,215) {\small $w_3$}
            \put(-80,155) {\small $\mathcal{S}$}
        \put(-100,187) {\small strip}
           \put(-394,130) {\small parallel zone}
          \put(-213,23) {\small $w_1$}
          \put(-247,86) {\small $w_2$}
          \put(-328,3) {\small $z_0$}
          \put(-332,78) {\small $w_3$}
          \put(-118,-6) {\small $w_1$}
          \put(-90,-6) {\small $w_2$}
          \put(-50,-6) {\small $w_3$}
          \put(-183,58) {\small $\phi$}
            \put(-80,45) {\small $+\mathbb H$}
                        \put(-100,70) {\small half plane}
             \put(-394,3) {\small elliptic zone}
           \caption{\small{The  different types of zones: center, annular, parallel, and elliptic; and their corresponding image under the rectifying coordinates $\phi$: half infinite cylinder,  finite cylinder, strip, and half plane.  The points labeled  $w_0, \ldots, w_3$ are poles of $R$ (saddles) and  $z_0$ and $z_1$ are zeroes of $R$ (sinks, sources, centers or elliptic points).}  }
    \label{fig:zones}
    \end{figure}
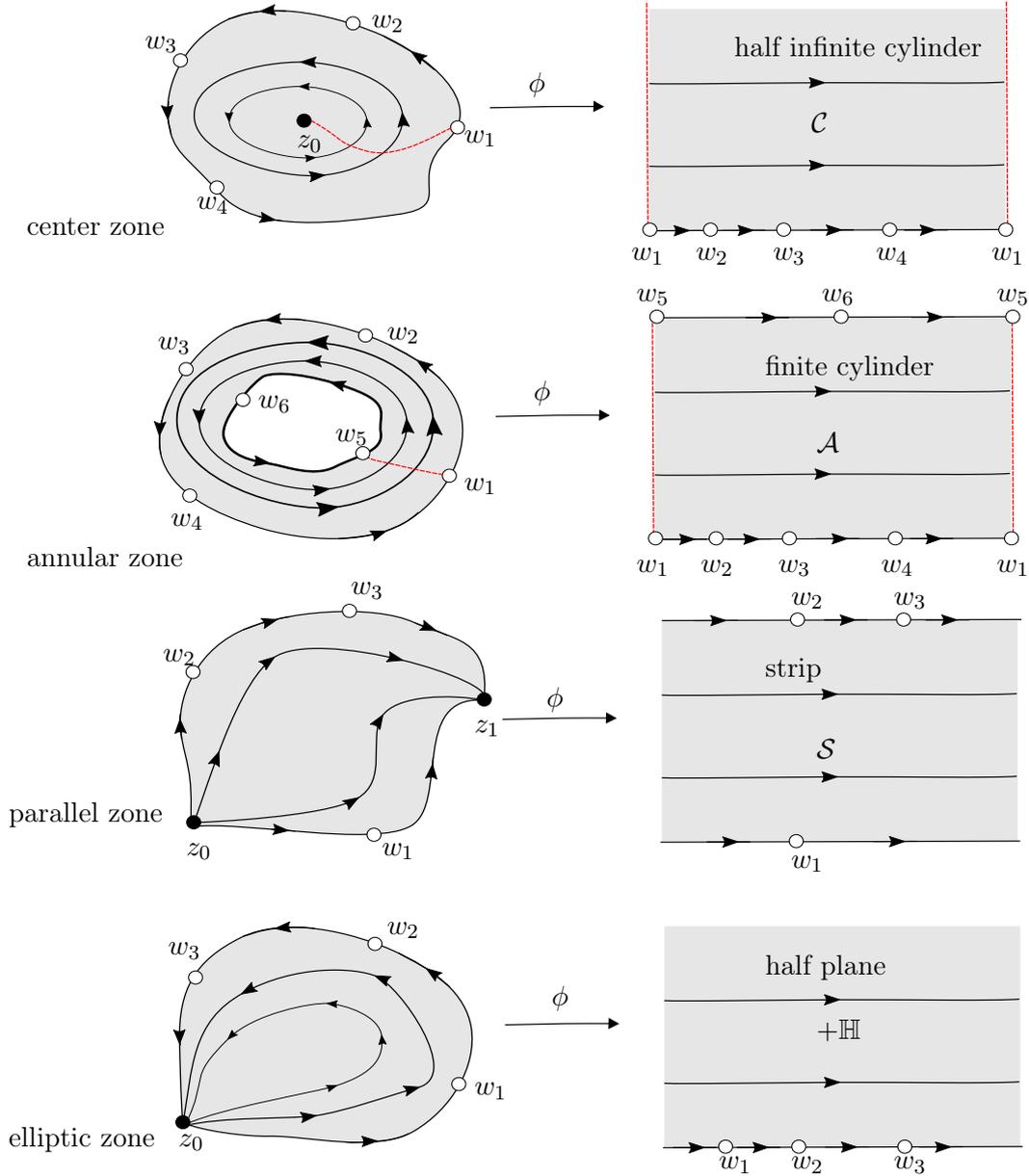

\begin{itemize}
\item A {\it center zone} $C$ contains  an equilibrium point $z_0$ which is a center. From the topological point of view $C\setminus \{ z_0\} $ is a doubly connected region conformally isomorphic to $\D \setminus \{0\}$.  Thus  $C \setminus \{z_0\}$ has infinite modulus $\mod(C\setminus \{ z_0\})=\infty$. Moreover, every center zone $C$  is foliated by periodic orbits of the same period except the equilibrium point $z_0$. See Figure \ref{fig:zones} first row.

\item An {\it annular zone} $A$  is   a doubly connected region with finite conformal modulus $0<\mod(A)<+\infty$ also foliated by periodic orbits of the same period. Each boundary component of $A$ is formed by one or several homoclinic/heteroclinic connections.  The value $\mod(A)$ is a conformal invariant of the double connected region $A$. However, deforming continuously the separatrix graph we can modify the value of  $\mod(A)$ to any other value in $(0,+\infty)$. See Figure \ref{fig:zones} second row.

\item A {\it parallel zone} $P$ is a simply connected region containing two different equilibrium points on the boundary of $P$, denoted by  $z_0$ and  $z_1$, corresponding to  the $\alpha-$limit point and $\omega-$limit point for all trajectories in $P,$ respectively. The boundary of $P$ contains one or two incoming landing separatrices and one or two outgoing landing separatrices, and zero or more homoclinic/heteroclinic connections. See Figure \ref{fig:zones} third row.

\item An {\it elliptic zone} $E$ is a simply connected region containing exactly one equilibrium point $z_0$  on the boundary, which is both the $\alpha-$limit and $\omega-$limit for all trajectories. The boundary of $E$ consists of one incoming landing separatrix, one outgoing landing separatrix, and zero or more homoclinic/heteroclinic connections. See Figure \ref{fig:zones} fourth row.

\end{itemize}

In the case of  a polynomial vector field only center, elliptic, and parallel  zones exist (\cite{ClassificationComplexPolynomial}) since annular zones needs at least two poles of $R$ and polynomial vector fields only have a unique pole.

In each zone we can define the rectifying coordinates that globally conjugate  the vector field $\xi_R$ to the unit vector field $ 1 (d/dz)$. In any simply connected domain avoiding zeros of $R$, the differential $\frac{dz}{R(z)}$ has an antiderivative unique up to addition by a constant
\[
\phi(z)=\int_{r_0}^z \frac{dw}{R(w)}.
\]

Note that
\[
\phi_* (\xi_R) = \phi'(z) R(z) \frac{d}{dz} = \frac{d}{dz}.
\]

The coordinates $w=\phi(z)$ are, for this reason, called {\it rectifying coordinates}. We will call the images of zones under rectifying coordinates {\it rectified zones}. The rectified zones are of the following types. See Figure \ref{fig:zones}.

\begin{itemize}
\item The image of a center zone (minus a curve contained in the zone which joins the center $z_0$ and a saddle equilibrium point in the boundary) under $\phi$ is a {\it  half infinite cylinder} $\mathcal C$. It could be either  an upper half infinite cylinder or a lower half infinite cylinder depending on the orientation of the closed trajectories in the center zone.  In Figure \ref{fig:zones} (first row) we show a center with the trajectories oriented counterclockwise and their corresponding upper half infinite cylinder.
\item The image of an annulus zone (minus a curve contained in the zone which joint two saddles  equilibrium points each one in the two different boundaries) under $\phi$ is a {\it  finite cylinder} $\mathcal A$. It could be either an upper  finite cylinder  or lower  finite cylinder depending  on the orientation of the trajectories in the annulus. In Figure \ref{fig:zones} (second row) we show an annulus with the trajectories oriented counterclockwise and their corresponding upper finite cylinder.
\item The image of a parallel zone under $\phi$ is a horizontal {\it strip} $\mathcal S$. See Figure \ref{fig:zones} (third row).
\item The image of an elliptic zone under $\phi$ is either an upper {\it half plane} $+\mathbb H$ or a lower half plane $-\mathbb H$. In Figure \ref{fig:zones} (fourth row) we show an elliptic zone  with the trajectories oriented counterclockwise and their corresponding upper half plane.
\end{itemize}

We turn our attention to $\infty$; this point can be interpreted as the north pole of the Riemann sphere $\hat{\mathbb C}$. The usual approach to investigate the
local phase portrait of \eqref{eq:Main} near $\infty$ is to consider the change of variables $w=1/z$. Thus $z=\infty$ for \eqref{eq:Main} becomes
 $w=0$ in the new variable. The local behavior of $\infty$ mainly depends  on the degrees of the polynomials $P$ and $Q$. Using this approach the point of $\infty$ could be a critical or a regular point.  We recall that $z_0$ is a regular point if  the rational flow \eqref{eq:Main} near  $z_0$  is conformally conjugated to $\dot{z}=1$ near the origin.

 We normalize  \eqref{eq:Main}  so that the point at infinity is always a regular point. We claim  that $\infty$ is a regular point if and only if $deg(P)=deg(Q)+2$. To see the claim  we assume that  $\infty$ is a regular point or  equivalently we assume that  the index at $\infty$ is equal to 0. From the Poincar\'e Index Theorem (\cite{DifferentialEquationsDynamicalSystems}, Thm. 8, Sec. 3.12) we have that
\[
\sum_{ c } I(c) = \chi (\CC)=2
\]

\noindent where the sum is taken over the critical points of the rational function $R(z)=P(z)/Q(z)$. Assuming the $\infty$ is not a critical point and using Proposition \ref{pro:index} we have that $\sum_{c}I(c)=n-m$, since every zero of $R$ has index equal to their multiplicity as a root of $P$ and every pole of $R$ has index equal to minus their multiplicity as a root of $Q$. Concluding thus that when $\infty$ is not a critical point then $n=m+2$. In the next proposition we prove that  we can always assume that $\infty$ is a regular point.

\begin{pro}\label{prop:infinity}
 Any rational flow  $R \frac{d}{dz}$  can be conformally conjugated by a Mobius transformation to one with $deg(P)=deg(Q)+2$.
\end{pro}
\begin{proof}
  Pick any regular point $\alpha$ for $R$ not on a separatrix and such that $P(\alpha) Q(\alpha) \neq 0$, we also select $a \neq 0$  and let $M(z)=\frac{az}{z-\alpha}$ be the Mobius transformation which sends $\alpha \mapsto \infty$ and $\infty \mapsto a$. Then $(M)_{\ast}(R(z)\frac{d}{dz})=\tilde{R}(w)\frac{d}{dw}$, where
\begin{align*}
  \tilde{R}(w)= & M'(M^{-1}(w))\cdot R(M^{-1}(w))  \\
  = &  \frac{-a \alpha}{\left(\frac{\alpha w}{w-a}-\alpha \right )^2}\cdot \frac{a_n\left( \frac{\alpha w}{w-a} \right) ^n + \cdots +a_0}{b_m\left( \frac{\alpha w}{w-a} \right) ^m + \cdots +b_0}
\end{align*}
and multiplying numerator and denominator by $(w-a)^m(w-a)^n$ gives
\begin{align*}
 \tilde{R}(w)  = &  - \frac{\left({w-a} \right) ^2} {a \alpha}\cdot \frac{(w-a)^m}{(w-a)^n}\cdot \frac{a_n (\alpha w)^n+a_{n-1}(w-a)^1(\alpha w)^{n-1} + \cdots +a_0 (w-a)^n}{b_m (\alpha w)^m+b_{m-1}(w-a)^1(\alpha w)^{m-1} + \cdots +b_0 (w-a)^m} \\
  & =- \frac{\left({w-a} \right) ^{m+2}} {a \alpha \left( {w-a} \right)^{n}}\cdot \frac{ P(\alpha) w^n + \mathcal O (w^{n-1}) }{ Q(\alpha) w^m + \mathcal O (w^{n-1}) },
\end{align*}
where you can see that the degree of the numerator is $n+m+2$ and the degree of the denominator is $n+m$.
\end{proof}

Hereafter and without loss of generality we can assume that  our rational flow given by \eqref{eq:Main} is such that $deg(P)=deg(Q)+2$.
\begin{rem}
Given $P$ a polynomial we have that  $\dot{z}=P(z)$  is conformally conjugate, via a Mobius transformation, to $\dot{z}=\tilde{P}(z)/(z-a)^{n-2}$
where $deg(P)=deg(\tilde{P})$. So, polynomial flows write as $R=P/Q$ with $deg(P)=deg(Q)+2$ and such that $Q$ has a unique root. In a similar way,
every Newton's flow $\dot{z}=- P(z)/P'(z)$ is conformally conjugate to $\dot{z} = (z-a) \tilde{P}(z) / \tilde{Q}(z)$ where $deg(\tilde{P})=deg(P)$ and $deg(\tilde{Q})=deg(P)-1$.
In both cases the finite parameter $a \in \mathbb C$ plays the role of $\infty$.
\end{rem}

As mentioned before, the Markus Theorem asserts that the separatrix graph and  one orbit in each canonical region is enough to
characterize the flow modulo  topological equivalences. The separatrix graph alone is not enough; the next lemma exemplifies that different rational flows can share the same separatrix graph.

\begin{lem}
The rational vector fields  $\dot{z}= z^3/(z-1)$ and $\dot{z}= i (z^2-1)(z^2-4)/(z^2+9)$  have an equivalent separatrix graph.
\end{lem}

\begin{proof}
We observe that in both cases  $\infty$ is  a regular point since $deg(P)=deg(Q)+2$.
We first consider the rational flow given by  $\dot{z}= \frac{z^3}{z-1}$ and we denote by $\Gamma_1$ its  separatrix graph. The phase portrait of   $\dot{z}= \frac{z^3}{z-1}$ exhibits an elliptic point  located at the origin of multiplicity $3$ and a saddle point at $z=1$ of multiplicity $1$. Moreover, In both cases the local phase portrait around each equilibrium point is given by four separatrices, two incoming and two outcoming.  We have that $\hat{\mathbb C} \setminus \Gamma_1$ is formed by four elliptic zones since all the other possibilities are excluded. So, every separatrix trajectory connects the saddle point with the multiple point and the complement is formed by four simply connected regions.

We secondly consider the rational flow given by  $\dot{z}= i \frac{(z^2-1)(z^2-4) }{z^2+9}$ and we denote by $\Gamma_2$ its separatrix graph. Simple computations show that this rational vector field has four centers located at $z = \pm 1$ and $z= \pm 2$ and two simple saddle points at $z= \pm 3i$. Moreover, every simple saddle point has four separatrix trajectories, two incoming and two outcoming. Thus, the phase portrait has four center zones and the boundary of every center zone is formed by two heteroclinic connections. So, every separatrix trajectory connects the two simple saddles.
 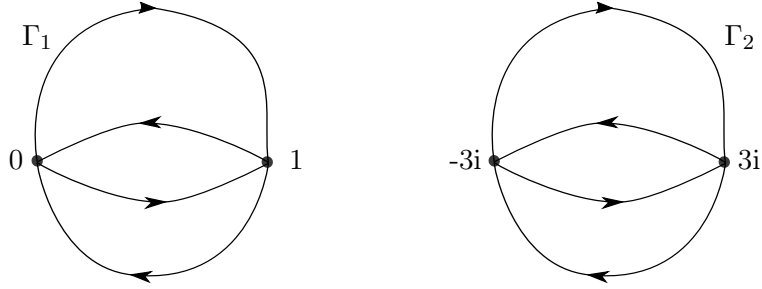
\begin{figure}[ht]
    \centering
     \includegraphics[width=0.55\textwidth]{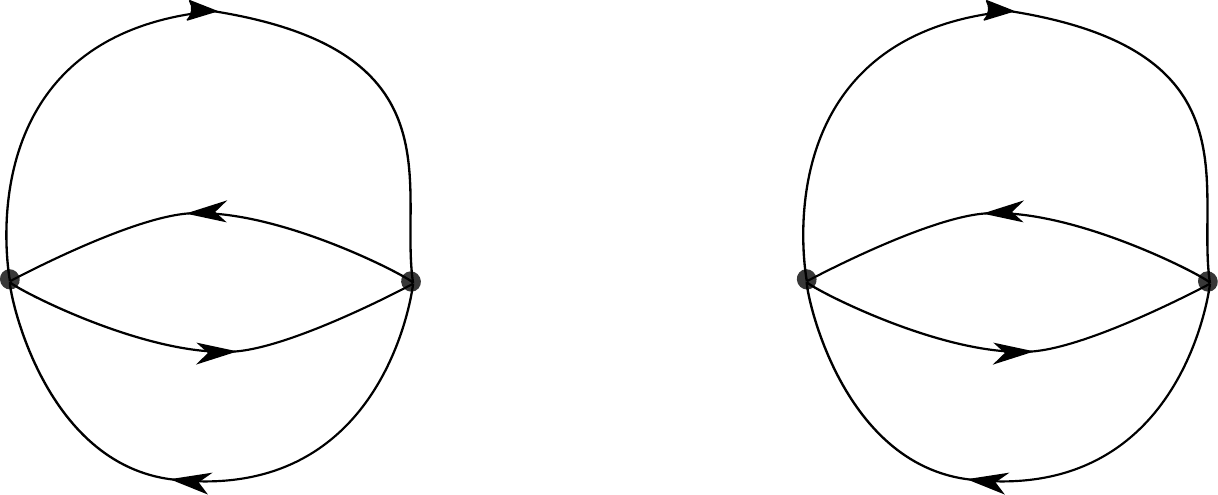}
    \put(-165,43) {\small 1 }
    \put(-270,43) {\small 0 }
    \put(-265,90) {\small $\Gamma_1$ }
    \put(-105,43) {\small -3i  }
        \put(3,43) {\small 3i  }
                \put(-2,90) {\small $\Gamma_2$  }
           \caption{\small{ Sketch of the separatrix graph of the rational vector fields $\dot{z}= z^3 / (z-1)$ (left)  and $\dot{z}= i (z^2-1)(z^2-4)/(z^2+9)$ (right).  In the first case the complement of separatrix graph  is formed by four elliptic zones (left)  and in the second case by four center zones (right). This demonstrates that the separatrix graph alone is not enough to determine topologically distinct phase portraits.}}
    \label{fig:example}
    \end{figure}

\end{proof}

In order to avoid this difficulty where different flows  share the same separatrix graph, we color the vertices of the separatrix graph rather than including an orbit in each region. More precisely,
we  distinguish between saddle points (white vertices)  and the rest of the critical points (black vertices).

\section{Definition of an admissible graph} \label{sec:admissible_graph}

In this section we will describe the conditions of an embedded planar directed graph to be the separatrix graph of a rational vector field. Roughly speaking an {\it admissible} graph is a planar and directed graph that looks like a separatrix graph of \eqref{eq:Main}.

We first introduce some notion of graphs. We denote the planar and directed graph by $\Gamma=\{V,E\}$, where $V$ is the set of vertices and $E$ the edges. In the set $V$ we distinguish between two kind of vertices: white and black vertices. We will see later that white vertices will correspond to saddle points  of the vector fields while black vertices will correspond to sink, sources and elliptic points of the vector field.

Given $ p \geq 0, \, q \geq 1$ and $k \geq 2$  we define the set of vertices and edges of $\Gamma$ by,
\begin{equation}\label{eq:vertices_edges}
\begin{array}{ll}
V & = \{b_1, \cdots, b_p, w_1, \cdots, w_q\} \\
E & = \{e_1, \cdots, e_k \, \}.
\end{array}
\end{equation}

\noindent where  every oriented edge $e_i=(x_i,y_i)$ starts at the vertex $x_i \in V $ and finish at the vertex $y_i \in V$, for $ 1 \leq i \leq k$;  an edge $e=(x,y)$ is called a {\it loop} in the case that $x=y$. For every vertex $x $ we define the {\it valence} $v(x)$  as the number of edges at the vertex  $x$. In the case that the graph  $\Gamma$ exhibits a loop $e=(x,x)$, then this loop contributes 2 to the valence of the vertex $x$.  In a similar way for every vertex $x$ we define the {\it cyclic reversals }$r(x)$ as the number of reversals, from edges starting at $x$ to edges finishing at $x$, when you made a turn around $x$. We notice that the valence counts the number of edges at the vertex, while the  cyclic reversal only counts the number of  incoming and outgoing edges at the vertex.  See Figure \ref{fig:valence_change}.

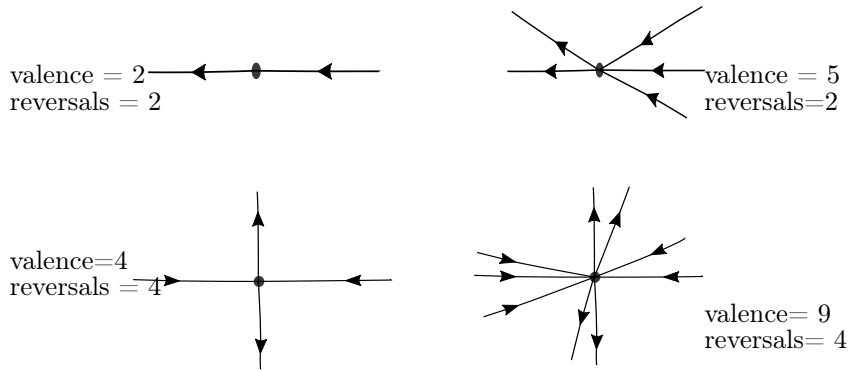
\begin{figure}[ht]
    \centering
     \includegraphics[width=0.45\textwidth]{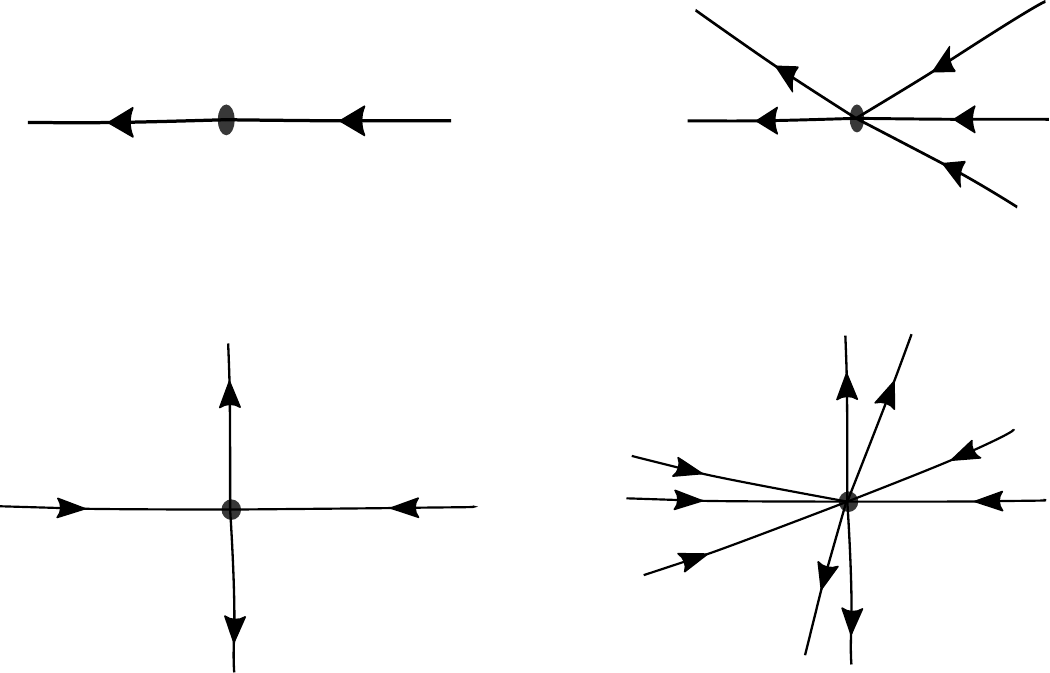}
    \put(0,108) {\footnotesize valence = 5 }
    \put(0,98) {\footnotesize reversals=2 }
   \put(-260,108) {\footnotesize valence = 2  }
   \put(-260,98) {\footnotesize reversals = 2  }
      \put(-260,38) {\footnotesize valence=4 }
      \put(-260,28) {\footnotesize reversals = 4}
     \put(0,18) {\footnotesize valence= 9 }
      \put(0,8) {\footnotesize reversals= 4}
           \caption{\small{Four examples counting  the  valence and  reversals of a concrete vertex. }}
    \label{fig:valence_change}
    \end{figure}

Given a rational vector field \eqref{eq:Main} we easily have the following properties at the critical points.  Firstly, center points do not belong to the separatrix graph. Furthermore, if $z_0$ is a saddle point of order $k \geq 1$ then $v(z_0)=r(z_0)=2(k+1)$, if $z_0$ is a multiple root of $R$ of order $k\geq2$ then $r(z_0)=2(k-1)$ and finally if $z_0$ is a source or a sink of $R$ then  $r(z_0)=0$. We notice that a priori we do not know what is the valence of a concrete root (simple or multiple) of $R$ from its local behavior.

We also introduce the following quantities associated to the graph $\Gamma$. We denote by $\mathcal V = \displaystyle  \sum_{i=1}^q v(w_i)$ the total valence at the white vertices and $\mathcal R =\displaystyle  \sum_{i=1}^p r(b_i)$ the total cyclic reversals  at the black vertices.

 We can define an admissible graph $\Gamma$ as a planar and directed graph with the main properties of the separatrix graph of a rational vector field. We notice that the separatrix graph of a rational vector fields  involves three main ingredients. The first one is the local behavior at the vertices, the second one is the Poincar\'e index formula and finally what kind of domains are allowed at  the complement of $\Gamma$.

\begin{defn} \label{def:admissible}
Let $ p \geq 0, \, q \geq 1$ and $k \geq 2$ three natural numbers. We consider  $\Gamma$  a planar and directed graph $\Gamma = \{ V,E\}$ where $V = \{ b_1, \cdots, b_p, w_1, \cdots, w_q\}$ are the vertices and $E  = \{e_1, \cdots, e_k \, \}$ are the edges. We say that $\Gamma$ is an admissible graph if and only if the following conditions are satisfied

\begin{enumerate}[label=(\alph*)]
\item There are not isolated vertices.
\item Every edge is  incoming or outgoing  from a  white vertex.
\item Every white vertex $w$ verifies that $v(w)=r(w)$ is an even number greater than or equal to 4.
\item Every black vertex $b$ verifies that $r(b)$ is an even number greater than or equal to 0.
\item $\displaystyle  p + \frac{\mathcal R }{2}  \leq  2 -q + \frac{\mathcal V} {2}\qquad$  (Poincar\'e index formula)
\item  The complement of $\Gamma$ is formed by simply connected  and doubly connected regions.  If $N+1$ is the number of connected components of $\Gamma$, then  $\mathbb C \setminus \Gamma$ has $N $ annular regions.  Every annular region is doubly connected whose boundary is formed by white vertices, and both boundary components have the same orientation. See figure \ref{fig:ringregion}.
\item  There are $c:= \frac{ \mathcal V }{2}  -q +2 - \left( p + \frac{\mathcal R}{2}  \right) \geq 0 $ center regions. Every center region is simply connected and  has boundary which is an oriented cycle formed by white vertices.  See figure \ref{fig:centerregion}.  
\item There are $\mathcal R$ elliptic regions. Every elliptic region is  simply connected  and whose boundary is a cycle of exactly one black vertex and one or several white vertices, oriented from the unique black vertex to itself. See  figure \ref{fig:halfregion}.

 \item The rest of components of $\mathbb C \setminus \Gamma$ are formed by parallel regions.  Every parallel region is simply connected  whose boundary contains two black vertices and two oriented paths (not necessarily disjoint) both going from one black vertex to the other black vertex.   See figure \ref{fig:stripregion}.
\end{enumerate}

\end{defn}

\begin{lem}\label{lem:admissible_rat}
Let $\Gamma$ be the separatrix graph of a rational vector field \eqref{eq:Main}, then $\Gamma$ is an admissible graph.
\end{lem}

\begin{proof}
We denote by $\Gamma$ the separatrix graph of $\dot{z}=R(z)$. We color the  saddle points  white and sinks, sources, and multiple points  black.
We assume that the rational flow given by $R$ writes as  $P(z)/Q(z)$, with $deg{(P)}=n$ and $deg{(Q)}=n-2$ (see Proposition \ref{prop:infinity})  and $n \geq 3$ (see Remark \ref{rem:degree2}). Since $deg(Q) \geq 1$, then we have at least one white vertex, $q \geq 1$, and at least two edges, $k \geq 2$. The first four properties of the definition of an admissible graph are trivially satisfied from the local behavior at critical points (see Proposition \ref{pro:index}).

We derive property (e) from Poincar\'e index formula.  Firstly,  we deal with saddle points (white vertices).  By assumption every saddle point corresponds to a root of the polynomial  $Q$. Writing
\[
Q(z)= \prod_{i=1}^q (z-w_i)^{d_i},
\]
\noindent  we have that
\begin{equation} \label{eq:PoincareBendixon1}
 \mathcal V =\displaystyle \sum_{i=1}^q v(w_i)=  \sum_{i=1}^q 2 (d_i +1 ) =  2 (n-2 + q),
\end{equation}
since a root $w_i$ of multiplicity $d_i$ has a valence equal to $2(d_i+1)$ (see Proposition \ref{pro:index}) and $deg(Q)=n-2$.

Secondly, we deal with sinks, sources and elliptic points (black vertices) which are roots of $P$. However, in that case we need to take into account center points (if there are any) since they are  roots of $P$  but not vertices of $\Gamma$. The number of sinks, sources and elliptic points verify
\begin{equation} \label{eq:PoincareBendixon2}
 \sum_{i=1}^p \left ( \frac{r(b_i)}{2} +1 \right )  = \frac{\mathcal R}{2} + p  \leq n,
\end{equation}
since they are roots of $P$ and $\deg(P)=n$. Furthermore, we have that the number of center zones is exactly $n- \frac{\mathcal R}{2} - p$.

Finally, combining Equations \eqref{eq:PoincareBendixon1} and \eqref{eq:PoincareBendixon2} we obtain  the desired inequality
\[
 p +  \frac{\mathcal R}{2}   \leq 2-q  + \frac{\mathcal V}{2}.
 \]

The rest of the properties follow from the four types of zones discussed in \S \ref{sec:preli}.  We only check the number of center regions (g). From Equation \eqref{eq:PoincareBendixon2} the number of center is equal to $n- \left( p +  \frac{\mathcal R}{2}  \right)$. Replacing $n$ from Equation  \eqref{eq:PoincareBendixon1} we obtain that the number of center zones is exactly
\[
c:= \frac{\mathcal V}{2} -q +2 - \left( p +  \frac{\mathcal R}{2}  \right).
\]

\end{proof}

\section{ Characterization of an admissible graph as a separatrix graph} \label{sec:admissible_separatrix}

The goal of this section is to prove Theorem A which states which planar graphs correspond to the separatrix graph of a rational flow. More precisely, Theorem A states that a planar and directed graph $\Gamma$ corresponds to the separatrix graph of a rational vector field if and only if it is an admissible graph. The precise definition of an admissible graph is given in \S \ref{sec:admissible_graph}. Almost trivially, any rational vector field must have separatrix graph satisfying these conditions, since the conditions were designed with the rational separatrix graph in mind (See Lemma \ref{lem:admissible_rat}).

Now we will show that the admissibility conditions are enough. The steps of the proof are listed here, and proving each step will follow.
\begin{enumerate}
\item
Build rectified zones from $\Gamma$, and glue these together to create a rectified surface $\M^{\ast}$ with punctures. See \S \ref{subsec:M_ast}.
\item
Construct an atlas for $\M^{\ast}$ to show that it is a Riemann surface, and use the charts around the punctures to define the closure $\M$. See \S \ref{subsec:M}.
\item
Use the Euler characteristic to show that $\M$ is homeomorphic to $\hat{\mathbb{C}}$, and the Uniformization Theorem then gives that it is, in fact, isomorphic to $\hat{\mathbb{C}}$. See \S \ref{MisoC}.
\item
Endow $\M^{\ast}\setminus \{w_1,\dots,w_q \}$ with the vector field $\frac{d}{dz}$ in the natural charts, and extend this holomorphically to the vector field $\xi_{\M}$, holomorphic on $\M \setminus \{w_1,\dots,w_q \}$ and meromorphic on $\M$.  See \S \ref{assvfs}.
\item With all these in our hands we can prove Theorem A. We show that the induced vector field   must be  a rational  vector field of the form \eqref{eq:Main} with $\deg(P)=\deg(Q)+2$. See \S \ref{finalpf}.

\end{enumerate}

\subsection{ The rectified surface $\M^{\ast}$ with punctures} \label{subsec:M_ast}

\subsubsection{Building zones from the metric graph}
Let $\Gamma = \{ V,E\}$ be an admissible graph with $q$ white vertices $w_1,  \ldots, w_q$, $p$ black vertices $b_1, \ldots, b_q$; and $k$ edges denoted by $e_1, \ldots,e_k$ (see \S \ref{sec:admissible_graph} for details).

We will  construct the four different types of rectified zones as a building blocks of the rectified surface $\M^{\ast}$:  {\it half-infinite cylinders} $\mathcal{C}$, {\it finite cylinders} $\mathcal{A}$,  {\it half-planes} $\pm \mathbb{H}$, and {\it strips} $\mathcal{S}$. We will first explain their topology, and afterwards explain their geometry, since it turns out that choosing the lengths of the edges joining white vertices on the boundaries is not as straightforward as one would initially imagine.\par

For each simply connected component of $\hat{\mathbb{C}}\setminus \Gamma$ which is bounded by a counterclockwise (resp. clockwise) oriented cycle with $N$ edges connecting $M$ white vertices, construct an upper (resp. lower) half-infinite cylinder $\mathcal{C}=\mathbb{H}_{\pm} / L \mathbb{Z}$, where the oriented cycle is on $\mathbb{R} / L \mathbb{Z}$ and $L$ is the sum of the $N$ segment lengths between the white vertices to be determined (see \S \ref{subsec:lengths}). Thus given any $L>0$ the above construction of the  half-infinite cylinder $\mathcal C$ depends on $L$.  See Figure \ref{fig:centerregion}.

\par
\begin{figure}[ht]
    \centering
     \includegraphics[width=0.85\textwidth]{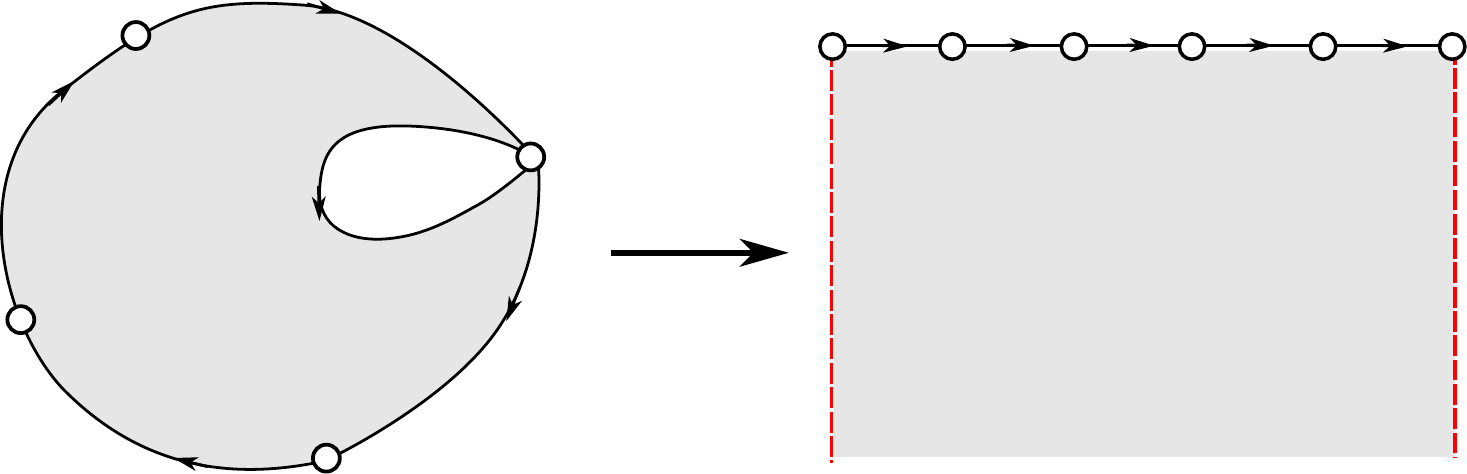}
    \put(-381,127) {$w_1$ }
    \put(-255,90) {$w_2$ }
    \put(-320,-7) {$w_3$ }
    \put(-417,40) {$w_4$ }
    \put(-180,125) {$w_1$ }
    \put(-150,125) {$w_2$ }
    \put(-115,125) {$w_2$ }
    \put(-85,125) {$w_3$ }
    \put(-50,125) {$w_4$ }
    \put(-10,125) {$w_1$ }
    \put(-130,75)  {half-infinite cylinder }
    \put(-95,45) {$\mathcal C$}
           \caption{ Example of a simply connected region bounded by a clockwise oriented cycle formed by  four white vertices  and five edges  (left), and the rectified zone $\mathcal C$ a half-infinite cylinder (right).}
    \label{fig:centerregion}
    \end{figure}
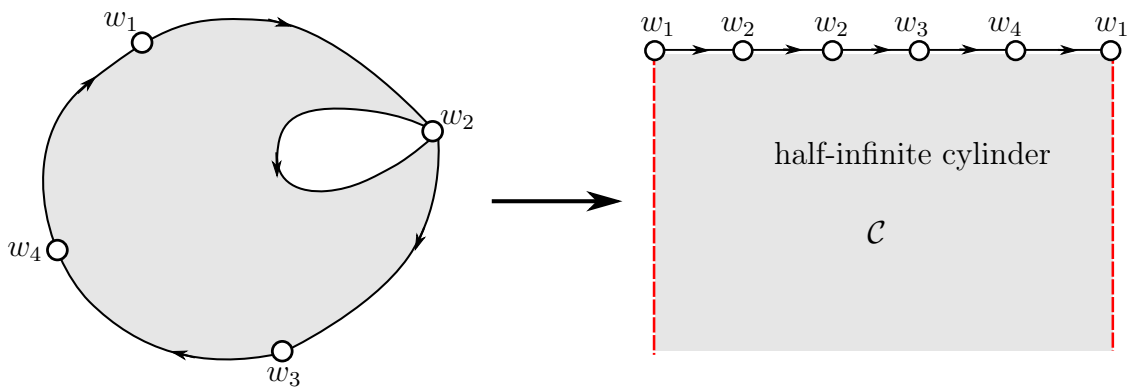
For each doubly-connected region, construct a finite cylinder $\mathcal{A}$ as follows (see Figure \ref{fig:ringregion}).
\begin{figure}[ht]
    \centering
     \includegraphics[width=0.85\textwidth]{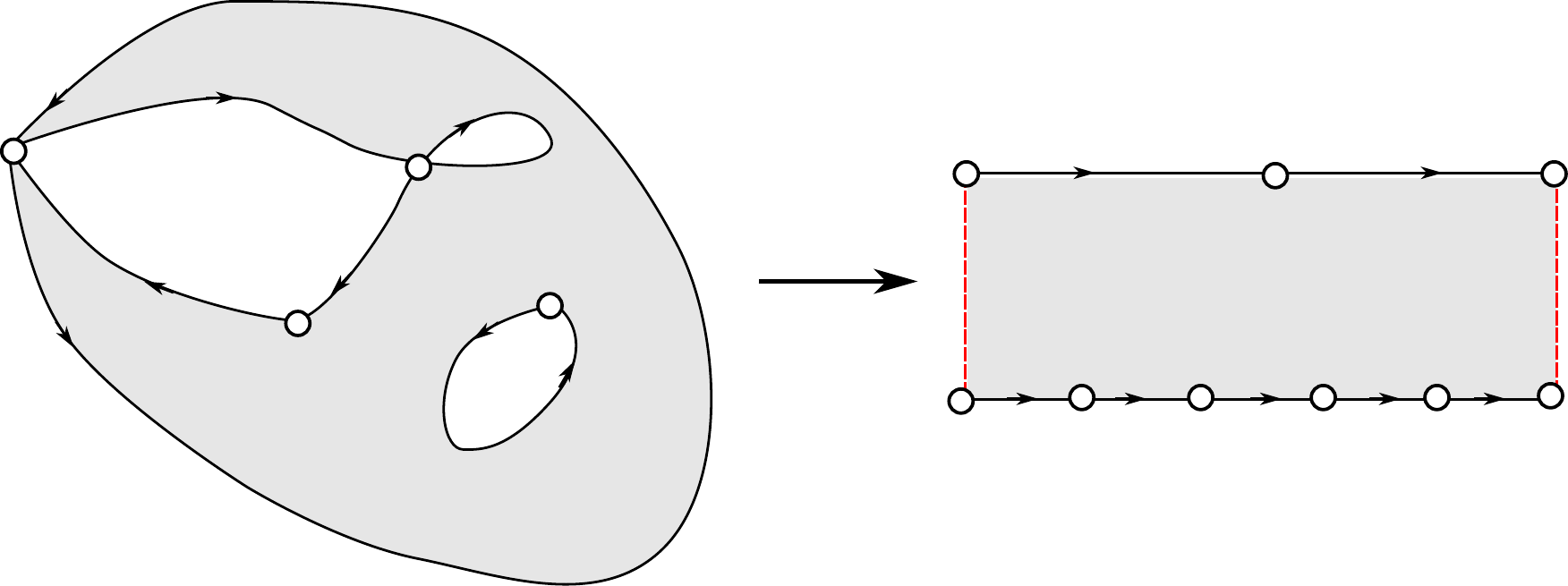}
    \put(-395,108) {$w_1$ }
    \put(-305,116) {$w_2$ }
    \put(-329,55) {$w_3$ }
    \put(-265,78) {$w_4$ }
    \put(-290,25) {$w_5$ }
    \put(-160,112) {$w_4$ }
    \put(-80,112) {$w_5$ }
    \put(-10,112) {$w_4$ }
    \put(-160,38) {$w_1$ }
    \put(-130,38) {$w_2$ }
    \put(-100,38) {$w_2$ }
    \put(-70,38) {$w_3$ }
    \put(-40,38) {$w_1$ }
    \put(-10,38) {$w_1$}
     \put(-85,84) {$\mathcal A$ }
          \put(-110,70) {finite cylinder }
           \caption{ Example of a doubly connected component bounded by two clockwise oriented cycles, one formed by two white vertices and two edges and the other formed by four vertices and five edges (left), and a corresponding rectified zone $\mathcal A$ a finite cylinder (right). }
    \label{fig:ringregion}
    \end{figure}
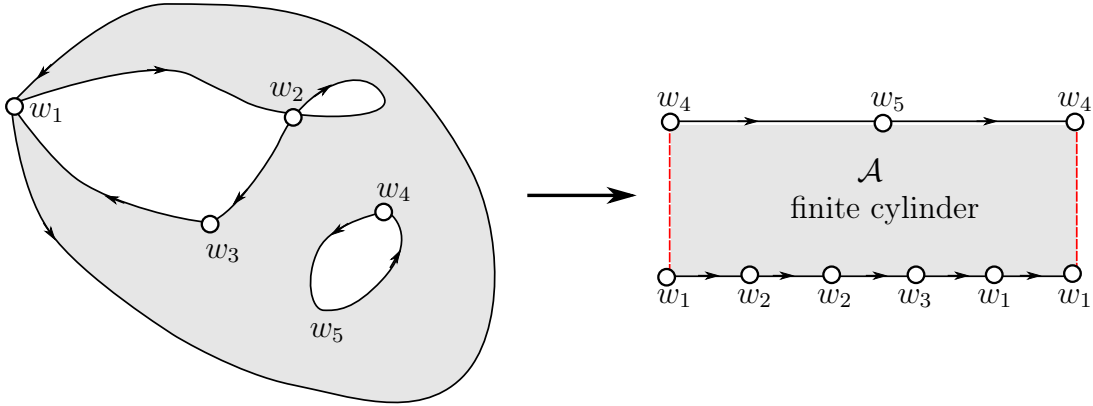
 Take the oriented boundary component which is an oriented cycle containing $N_1$ edges connecting  $M_1$ white vertices  
where the annulus is to the left of the boundary component and send it to $\mathbb{R} / L\mathbb{Z}$, where $L$ is the sum of the $N_1$ segment lengths between the white vertices to be determined (see \S \ref{subsec:lengths}).
Take the oriented boundary component where the annulus is to the right and send it to $\mathbb{R} /L \mathbb{Z}+i$, where $L$ is the same as for the other boundary component. This is a finite cylinder with height 1 and circumference $L$. As before this construction could be done for every $L>0$.   As we mention before, we can deform continuously the separatrix graph such that  $\mod (\mathcal A)= L$. Note that there is a choice of shear, the relative positions of the white vertices on the upper and lower boundaries.

\begin{figure}[ht]
    \centering
     \includegraphics[width=0.85\textwidth]{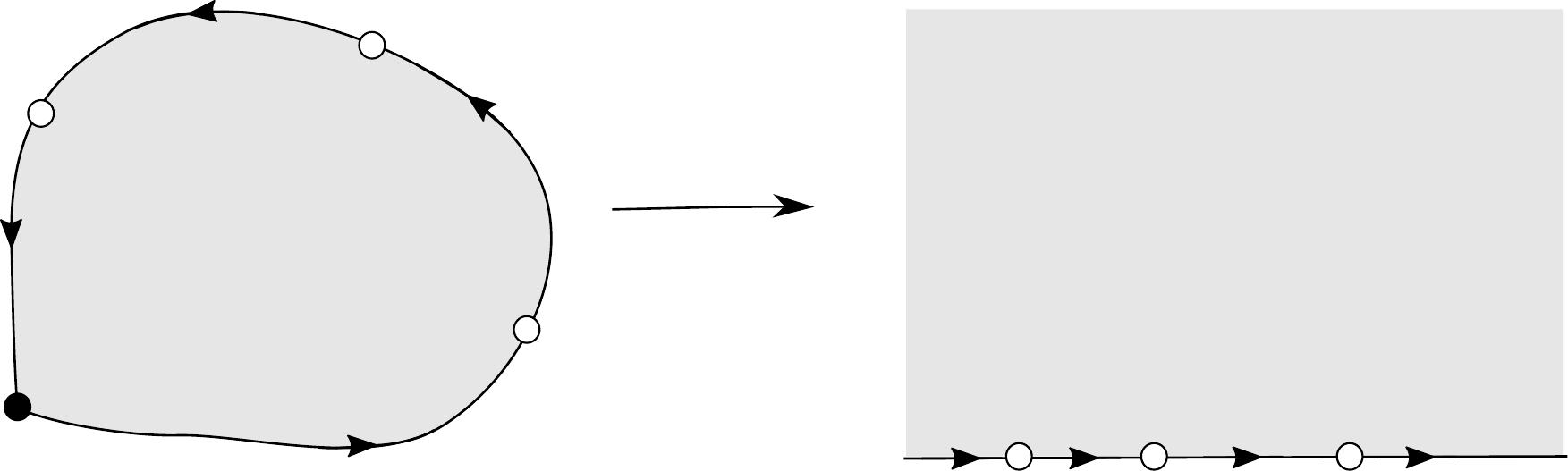}
    \put(-390,18) {$b_1$ }
    \put(-300,110) {$w_2$ }
    \put(-260,36) {$w_1$ }
    \put(-415,86) {$w_3$ }
    \put(-145,-7) {$w_1$ }
    \put(-109,-7) {$w_2$ }
    \put(-63,-7) {$w_3$ }
        \put(-90,55) {$\mathbb H_+$}
                \put(-100,70) {half-plane}
           \caption{Example of simply connected component with exactly one black vertex on the boundary and  three  white vertices forming a counterclockwise cycle with four edges (left), and their corresponding rectified zone $\mathbb H_+$ a half-plane (right).}
    \label{fig:halfregion}
    \end{figure}
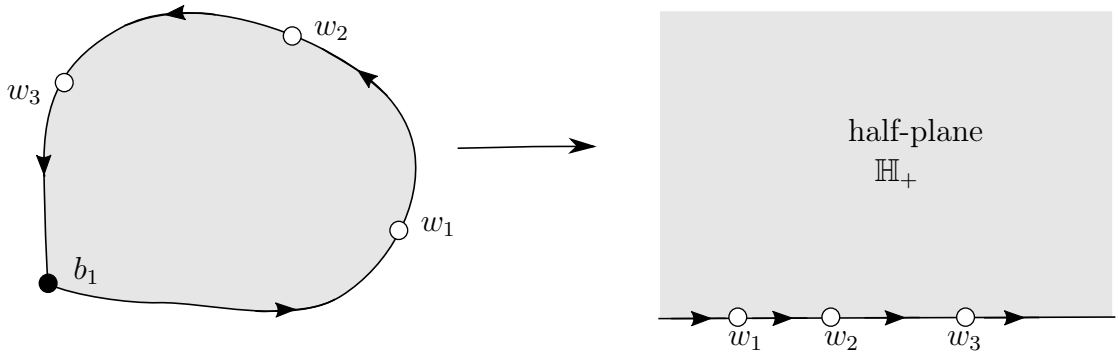

For each simply connected component with boundary oriented from the unique black vertex to itself, map this boundary to the real line such that the orientation of the boundary is from $-\infty$ to $+\infty$ and such that  $\pm \infty$ correspond to the black vertex. Take the upper half-plane $\mathbb{H}_+$ (resp. lower half-plane $\mathbb{H}_-$) if the component is to the left (resp. right) of its boundary.\par
For each simply connected component with exactly two black vertices on the boundary, the boundary consists of two sets of edges (not necessarily disjoint) which connect one black vertex to the other, respecting the orientation of the edges (see Figure \ref{fig:stripregion}).
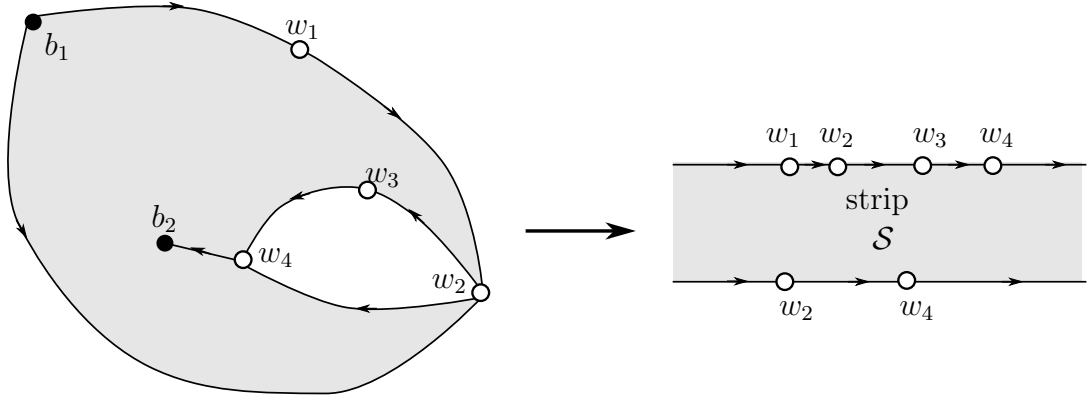
\begin{figure}[ht]
    \centering
     \includegraphics[width=0.85\textwidth]{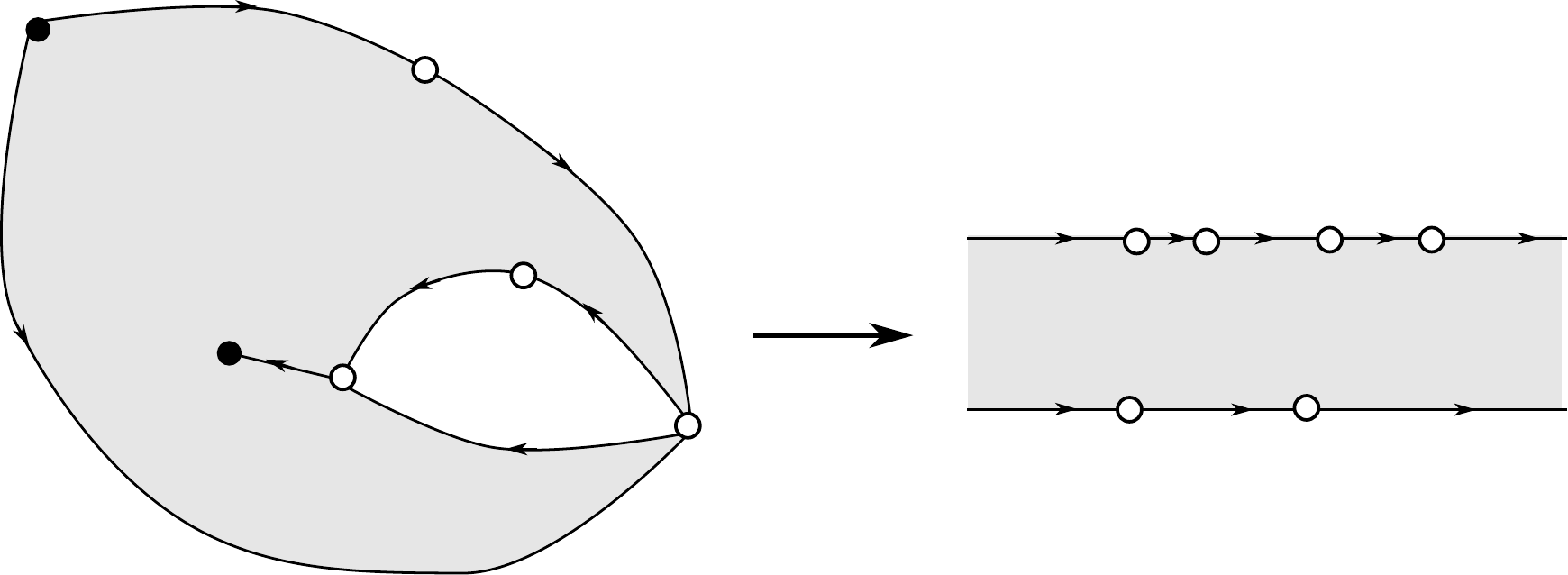}
    \put(-390,128) {$b_1$ }
    \put(-350,63) {$b_2$ }
    \put(-300,136) {$w_1$ }
    \put(-245,40) {$w_2$ }
    \put(-270,80) {$w_3$ }
    \put(-310,50) {$w_4$ }
    \put(-120,95) {$w_1$ }
    \put(-100,95) {$w_2$ }
    \put(-65,95) {$w_3$ }
    \put(-40,95) {$w_4$ }
        \put(-80,55) {$\mathcal S$ }
                \put(-90,70) {strip }
    \put(-115,30) {$w_2$ }
    \put(-70,30) {$w_4$ }
           \caption{Example of simply connected component with exactly two black vertices  and four white vertices  on the boundary, the boundary consists of two sets of edges  which connect one black vertex to the other. More precisely,  $ b_1 \to w_1 \to w_2 \to w_3 \to w_4 \to b_2$ and  $b_1 \to w_2 \to w_4 \to b_2$ (left); and the corresponding rectified zone $\mathcal S$ a strip (right).}
    \label{fig:stripregion}
    \end{figure}
Construct a horizontal infinite strip with height 1 such that the upper and lower boundaries are horizontal lines, each corresponding to the two boundary sets of edges. There is again a choice of  relative positions of the white vertices on the upper and lower boundaries. \par
\subsubsection{Assigning Lengths}\label{subsec:lengths}
We now discuss how to set the lengths of the heteroclinic and homoclinic segments on the boundaries of these regions. It would be simplest to set all of these lengths equal to 1, but this cannot (usually) be done, as will now be explained.  Each annular region (if they are any) has two boundary components, say one which has $N_1$ edges and the other has $N_2$ edges. The lengths of these boundary components need to be equal, since the trajectories in annular regions are isochronous. Since in general $N_1 \neq N_2$, setting each edge length equal to 1 would give one boundary component length $N_1$ and the other boundary component length $N_2$. Even though the lengths cannot usually all be set to equal 1, the following result shows that there does exist some consistent assignment of lengths for every admissible graph.
\begin{pro}
\label{lengthprop}
For every admissible graph, let $x_1,\dots,x_n$ be the lengths of the edges that connect white vertices to white vertices.  There exists assignment of positive numbers to $x_1,\dots,x_n$ such that for each of the $N$ annuli, the  lengths of each of the two boundary components are equal.
\end{pro}
The proof, relegated to the Appendix, will apply a result by Dines \cite{PositiveSolutionsLinear} regarding the existence of solutions to systems of linear equations which have all positive components.

\subsubsection{Definition of $\M^{\ast}$}
Utilizing admissible $\Gamma$ and Proposition \ref{lengthprop}, we construct half-infinite cylinders $\mathcal{C}_i$, finite cylinders $\mathcal{A}_i$, half-planes $\pm \mathbb{H}$, and strips $\mathcal{S}_i$ as subsets in $\mathbb{C}$. We define
\begin{equation}
\mathcal{M}^{\ast}=\left(\mathcal{C}_i\sqcup \mathcal{A}_i\sqcup \pm \mathbb{H} \sqcup \mathcal{S}_i \right)/\sim,
\end{equation}
where $\sim$ is the equivalence relation  such that all points corresponding to the same white vertex are identified, and each pair of points on the two occurrences of any separatrix (edge of $\Gamma$) are identified by isometry.\par
We will make charts of the neighborhoods of the boundary components of $\M^{\ast}$ to show that each corresponds to one point.  The natural $(p+c)$-point closure of $\M^{\ast}$ is denoted $\M$ and is called the \emph{rectified surface}. We notice that  these $p+c$ points correspond to the $p$ black points and the $c$ centers (see Definition \ref{def:admissible} (g))  that are omitted in  the construction of $\M^{\ast}$. The construction of the Riemann surface structure on $\M$ is contained in the construction below of an atlas.\par

\subsection{The Riemann surface $\M$} \label{subsec:M}

\subsubsection{An atlas for $\mathcal{M}$}\label{subsubsec:atlas}
We  show that $\M$ is a Riemann surface by constructing an atlas $(U_i,\eta_i)$ for $\M$ with holomorphic transition maps.
There are obvious charts over the interior of each half-infinite cylinders $\mathcal{C}_i$, finite cylinders $\mathcal{A}_i$, half-planes $\pm \mathbb{H}$, and strips $\mathcal{S}_i$, and the transition maps are, at worst, translations in $\mathbb{C}$.  It remains to show charts over:
\begin{itemize}
\item
points on edges  of $\Gamma$ (separatrices),
\item
the white vertices $\{ w_i \mid i=1,\dots,q\}$, and
\item
the punctures of $\M^{\ast}$.
\end{itemize}

We treat each case in turn. Firstly, for points on edges  we note the following. Since each edge is on the upper boundary of exactly one rectified zone and on the lower  boundary of exactly one rectified zone, we define a neighborhood of a point on a edge in the natural way: construct sufficiently small half-disks of the same radius in the upper and lower zone, and identify by isometry (see Figure \ref{fig:atlasseps}). \par
\begin{figure}[ht]
    \centering
     \includegraphics[width=0.3\textwidth]{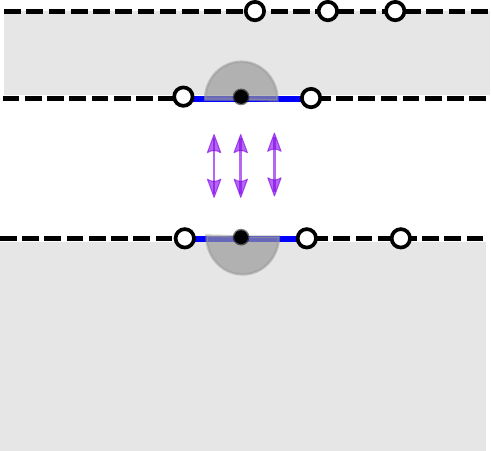}
         \put(-75,95) {{\footnotesize $z$}}
                  \put(-75,110) {{\small $D_1$}}
         \put(-75,66) {{\footnotesize $z$}}
                           \put(-75,48) {{\small $D_2$}}
           \caption{Sketch of the  construction of a neighborhood of a point $z$ in an edge of the admissible graph. }
    \label{fig:atlasseps}
    \end{figure}
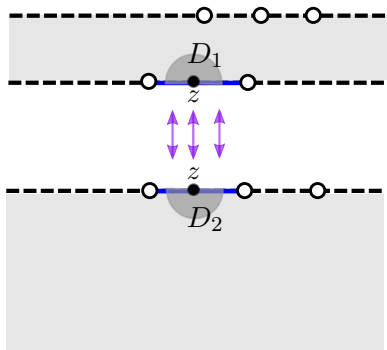

Secondly, let  $w:=w_i$  be one of the $q$ white vertices of $\Gamma$. We know that the valence at $w$ is equal to  $2(m+1)$ with $m \geq 1$, since $\Gamma$ is an admissible graph (see Definition \ref{def:admissible} (c)). Thus, $w$ is on the boundary of $2(m+1)$ rectified zones:  $m+1$ on lower boundaries of half-planes, strips, or (finite or half-infinite) cylinders, and $m+1$ copies on upper boundaries. We define a chart  around each $w$ as follows. Let $D^{\pm}_{k}$, $k=1,\dots,m+1$ be the upper or lower semi-disk with center $w$ in either a strip, half-plane, half-infinite cylinder, or
finite cylinder, $+$  if an upper semi-disk and $-$ if
a lower semi-disk, and with radius $r$ sufficiently small. The neighborhood  of $w$ is then the half-disks taken with  proper identification, $\{ D^{\pm}_{k} \}/\sim$, since this maps univalently to a sufficiently small open disk in $\mathbb{C}$ by an $(m+1)$-st root.
In Figure \ref{fig:atlaspoles} we  sketch the situation  where $m=1$ which maps to a small disk in $\mathbb{C}$ under a suitable square root.

  \par
\begin{figure}[ht]
    \centering
     \includegraphics[width=0.95\textwidth]{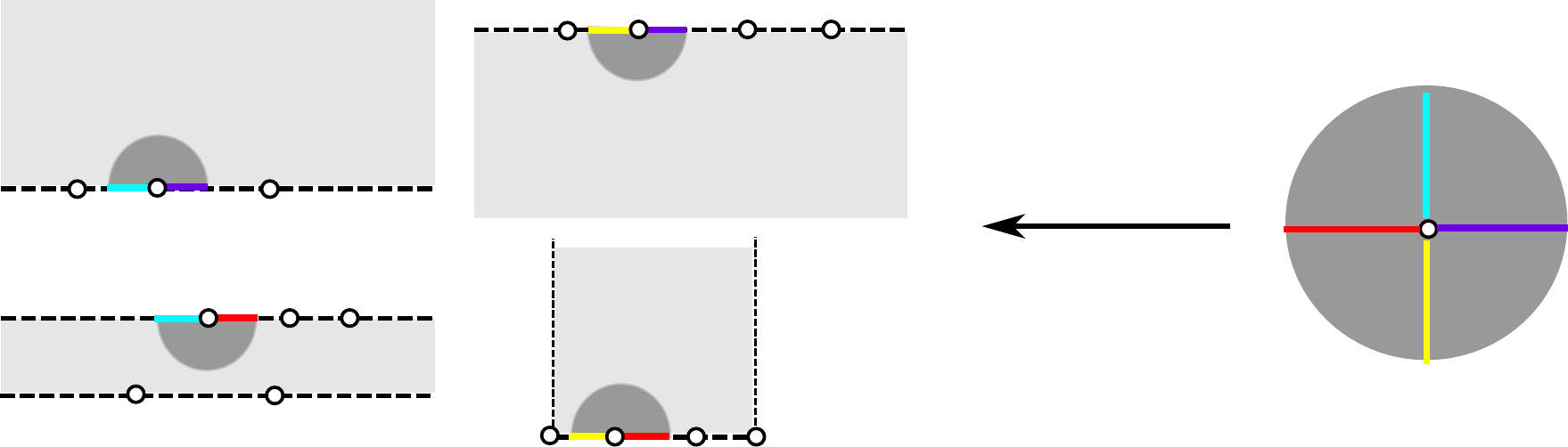}
    \put(-410,80) {$w$}
    \put(-410,95) {$D^+_1$}
    \put(-395,25) {$w$}
    \put(-380,20) {$D^-_1$}
    \put(-270,110) {$w$}
    \put(-265,100) {$D^-_2$}
    \put(-275,5) {$w$}
    \put(-255,10) {$D^+_2$}
    \put(-160,70) {$z \mapsto (z+w)^2$}
    \put(-30,75) {$D^+_1$}
    \put(-70,75) {$D^-_1$}
    \put(-70,40) {$D^+_2$}
    \put(-30,40) {$D^-_2$}
           \caption{ We sketch the construction of the neighborhood  of a  white vertex $w$ with valence  4, which corresponds to $m=1$. Thus $w$ is on the boundary of four rectified zones:  an upper half plane (top left), a strip  (bottom left),   a lower half plane (top right) and an upper half-infinite cylinder (bottom right). Suitable identification of these half-disks gives a neighborhood of $w$ in $\mathcal{M}$ which maps univalently to a disk in $\mathbb{C}$ under a suitable square root.}
    \label{fig:atlaspoles}
    \end{figure}
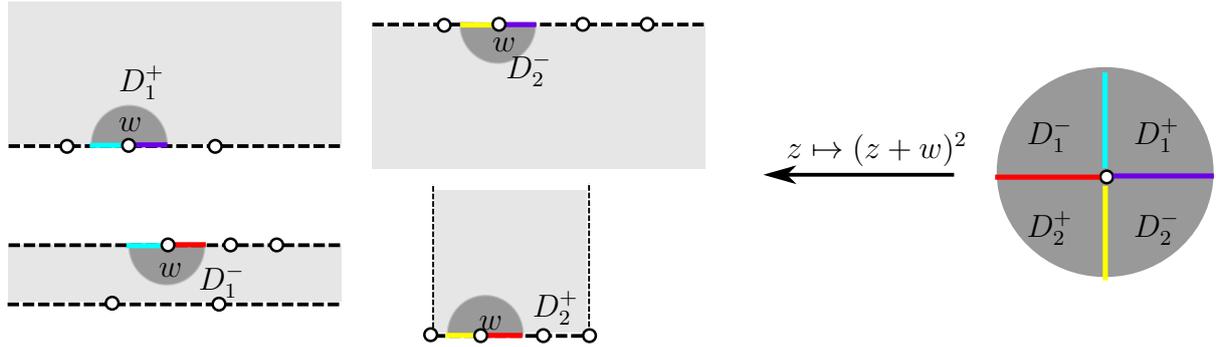
Thirdly, we now construct charts in neighborhoods $U \subset \M$ covering each boundary component of
$\M^{\ast}$.  The boundary components of $\M^{\ast}$ come from the black vertices of $\Gamma$ or the ends of half-infinite cylinders. The latter case is simpler, so we begin there. \par
 Each half-infinite cylinder $\mathcal{C}$ is conformally isomorphic to $\mathbb{D}\setminus \{ 0\}\subset \mathbb{C}$. This chart extends homeomorphically to the closure, by mapping the puncture $\infty_{[\mathcal{C}]}$   to $0\in \mathbb{C}$. We remember that there are $c$ cylinders and hence $c$ such punctures (see Definition \ref{def:admissible} (g)).\par
 Next, consider the boundary components of $\M^{\ast}$ which correspond to black vertices   in $\Gamma$   $\{ b_i \mid i=1,\dots,p\}$ which only have edges directed toward (resp. away from) the black vertex.  Let $b := b_i$ be one of the $p$ black vertices of $\Gamma$ with this property.  Each simply connected component with those edges on their boundaries must be a strip zone. Gluing those strips by the corresponding separatrices and truncating such that $\Re (z)>R$ (resp. $\Re (z)<R$) shows that a neighborhood $U_{[b]}$ of $b$ is a half-infinite cylinder to the right (resp. left). The conformal isomorphism $z \mapsto\exp\left( \frac{\mp 2\pi i}{\rho}z\right)$ maps this cylinder to a punctured neighborhood of $0\in \mathbb{C}$, where $\rho \in \mathbb{H}$ is the number that gives the height and shear of the identification. This local chart extends homeomorphically to its closure (see Figure \ref{fig:atlasstrips}). \par
\begin{figure}[ht]
    \centering
     \includegraphics[width=0.5\textwidth]{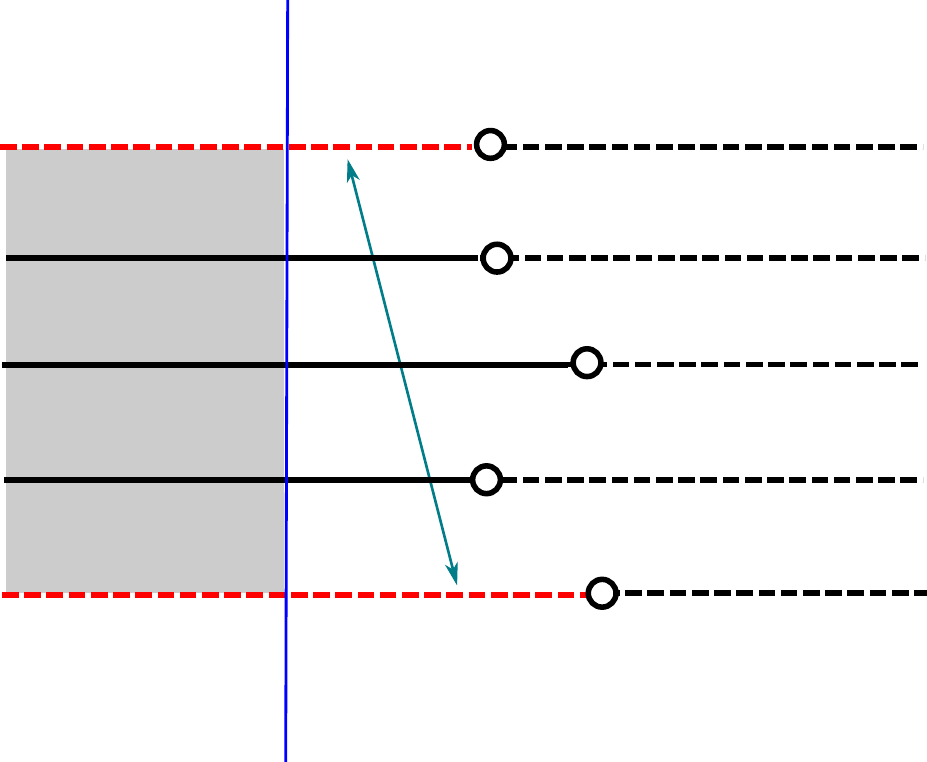}
    \put(-230,115) {$U_{[b]}$ cylinder}
    \put(-140,145) {$\rho \in \mathbb{H}$}
    \put(-160,20) {{\b $\Re(z)=R$}}
           \caption{ A neighborhood $U_{[b]}$ of $b:=b_i$, where $b_i \subset \Gamma$ is a black vertex with four edges in $\Gamma$ directed away from it.
             Gluing those strips by the corresponding separatrices and truncating gives that $U_{[b]}$  is a half-infinite cylinder to the left (the red dashed lines in the figure are identified by $\rho \in \mathbb{H}$). Therefore, $U_{[b]}$ is conformally isomorphic to $\mathbb{D}\setminus \{ 0\}\subset \mathbb{C}$, and the chart extends homeomorphically to the puncture.  }
    \label{fig:atlasstrips}
    \end{figure}
 The only case remaining is for a black vertex $b:=b_i$ which has adjacent edges directed both towards and away from it. There is a neighborhood $U_{[b]}$ of the boundary at infinity in rectifying coordinates that we will show  is doubly connected and has infinite modulus. Thus, showing the boundary is a single point.
 \begin{lem}
 \label{homotoplem}
 All simple, closed curves in $U_{[b]}$  are either homotopic to the bounded boundary component $\gamma_0$ or homotopic to a point.
 \end{lem}
 \begin{proof}
   If $\gamma_1$ is not homotopic to the boundary component $\gamma_0$, then there exists a simple curve $\gamma_2$ which joins $\gamma_0$ to infinity, where $\gamma_1$ and $\gamma_2$ do not intersect (see Figure \ref{fig:homotopic}).
 \begin{figure}[ht]
    \centering
     \includegraphics[width=0.6\textwidth]{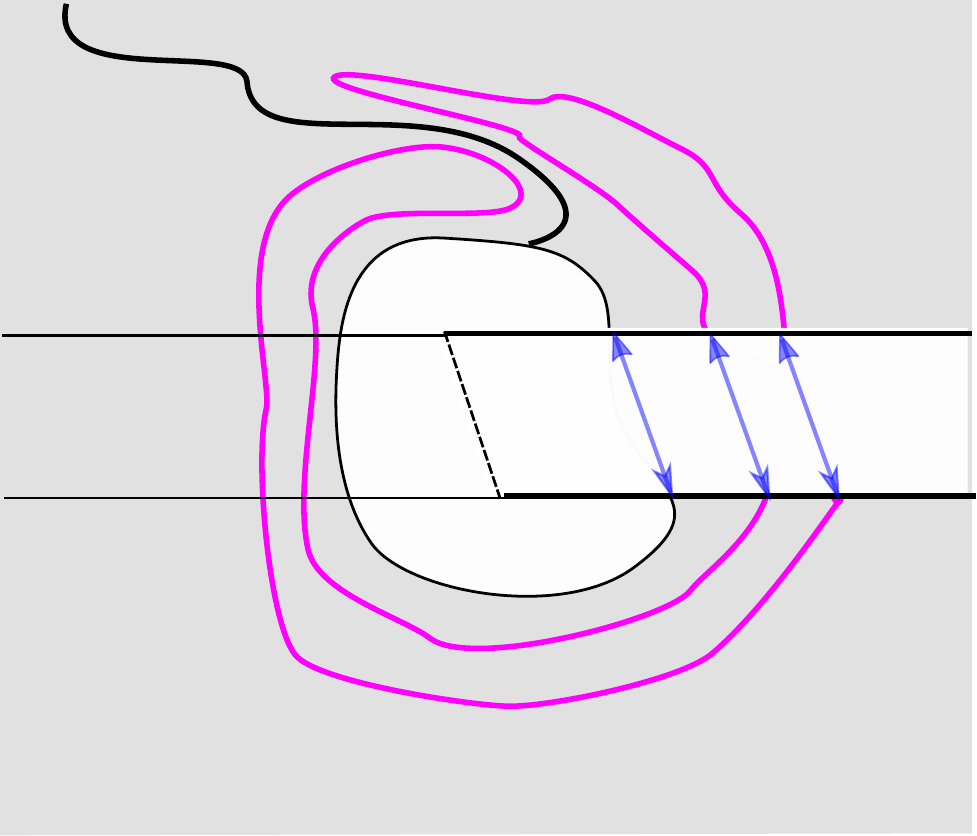}
    \put(-160,78) {$\gamma_0$ }
    \put(-280,228) {$\gamma_2$ }
    \put(-80,200) {$\gamma_1$ }
    \put(-35,228) {$U_{[b]}$  }
    \caption{If $\gamma_1$ is not homotopic to the boundary component $\gamma_0$, then there exists a simple  curve $\gamma_2$ which joins $\gamma_0$ to infinity, where $\gamma_1$ and the $\gamma_2$ do not intersect. The curve $\gamma_1$ must be homotopic to a point since it is entirely contained in the simply connected set $U_{[b]} \setminus \gamma_2$.}
    \label{fig:homotopic}
\end{figure}
   The set $U_{[b]} \setminus \gamma_2$ is simply connected. Therefore, $\gamma_1$ is homotopic to a point since it is contained entirely in a simply connected set.
 \end{proof}
 By Lemma \ref{homotoplem}, $U_{[b]}$ is doubly connected since every non-trivial, simple closed curve in $U_{[b]}$ is homotopic to $\gamma_0$.  We now use Gr\"{o}tzsch's inequality to show that this doubly connected region has infinite modulus.  Indeed, it is easy to see in rectifying coordinates that one can construct infinitely many disjoint annuli  contained in $U_{[b]}$ and homotopic to $\gamma_0$ which have modulus bounded away from zero since the $\mathbb{H}$ and strips are unbounded. Therefore, $\mod\left(U_{[b]}\right)=\infty$ since Gr\"{o}tzsch's inequality gives $\mod\left(U_{[b]}\right) \geq \sum_{i=1}^{\infty}\mod (A_i)$, where the latter must be infinite since each modulus $\mod (A_i)$ is bounded away from zero. Therefore, $U_{[b]}$   is conformally isomorphic to $\mathbb{D}^{\ast}$  and can be extended homeomorphically by mapping the puncture $\infty_{[b]}$ to $0\in \mathbb{C}$.
\par

Once we have defined a neighborhood $U$ for every point in  $\mathcal M$, we notice that all transition maps are translations or compositions of translations with conformal isomorphisms.
Therefore, we have made an atlas  with holomorphic transition maps on the Hausdorff space $\M$, and we can conclude
that $\M$
is a Riemann
surface.
\subsection{$\M$ isomorphic to $\CC$}
\label{MisoC}
In this section we prove the following result.
\begin{pro}\label{homeoprop}
$\M$ is homeomorphic to $\CC$.
\end{pro}
We first need a lemma regarding the Euler characteristic.
\begin{lem}\label{Eulerlemma}
  An embedded planar graph with $N+1$ connected components separated by $N$ annular regions satisfies $v-e+f=2$, where $f$ is the number of simply connected faces (including the face containing $\infty$).
\end{lem}
\begin{proof}
It is well known that $v-e+f=2$ holds for connected graphs embedded in the plane with $v$ vertices, $e$ edges, and $f$ faces which are simply connected. If there are additionally $N$ annular regions, one may introduce one edge per annulus that connects a vertex on one boundary component of the annulus to the other boundary component.  This yields a connected embedded planar graph which has $v$ vertices, $e+N$ edges, and $f+N$ faces. Substituting in the equation gives the same result.
\end{proof}
\begin{proof}[Proof of Proposition \ref{homeoprop}]
We utilize the Euler characteristic. Note that
$\M$ has a topology induced by the topology of $\C$. The identifications in constructing $\M$ give a graph on $\M$ with the same topology as $\Gamma \subset \mathbb{C}$. By Lemma \ref{Eulerlemma}, $\Gamma$ satisfies $v-e+f=2$, where $f$ counts the exterior face but not the annular regions.  Hence, $\M$ is homeomorphic to $\CC$.
\end{proof}
\begin{cor}
$\M$ is isomorphic to $\CC$.
\end{cor}
\begin{proof}
Since $\M$ is a simply connected Riemann surface that is
compact, it must be conformally equivalent to the Riemann sphere $\CC$ by The
Uniformization Theorem.
\end{proof}
\subsection{The Vector Field $\xi_{\M}$}
\label{assvfs}
We will show in this section that the singularities of $\xi_{\M}$ are equilibrium points and poles. The vector field $\xi_{\M}$ is defined on the half-infinite cylinders $\mathcal{C}$, finite cylinders $\mathcal{A}$,  half-planes $\pm \mathbb{H}$, strips $\mathcal{S}$, and neighborhoods of separatrix points as $\left(\eta\right)_{\ast}\left(\xi_{\M}\right)=\vf$. It remains to show what the extension of the vector field is at the punctures and at the white vertices.\par
Recall from \S \ref{subsubsec:atlas} that a neighborhood $U_{[b]}$ of a puncture $\infty_{[b]}$ of $\xi_{\M}^{\ast}$ that corresponds to a black vertex $b$ of $\Gamma$ with incoming only or outgoing only edges corresponds to a truncated union of strips identified on the boundary to form a cylinder (recall also Figure \ref{fig:atlasstrips}).
 There is a chart $\eta: U_{[b]} \rightarrow V$, where $V$ is a neighborhood of $0$, such that
\begin{equation}
\left(\eta\right)_{\ast}\left(\xi_{\M}\right)=\left(\varphi\right)_{\ast}\left(\vf\right)=\frac{\mp 2 \pi i}{\rho}z \cdot \vf,
\end{equation}
where $\varphi(z)=\exp\left( \frac{\mp 2 \pi i}{\rho} z\right)$.  The extension of the chart at the puncture $\infty_{[b]}$ leads to a holomorphic extension of $\xi_{\M}$ at the puncture  by $\xi_{\M}=0$.\par
A neighborhood $U_{[\mathcal{C}]}$ of a puncture $\infty_{[\mathcal{C}]}$ stemming from a cylinder zone is nearly identical to the above. The total length $L$ of homoclinics on the boundary gives the map $\varphi(z)=\exp\left( \frac{\mp 2 \pi i}{L} z\right)$ giving a vector field which extends to $\xi_{\M}=0$ at the puncture. \par
 A neighborhood $U_{[w]}\subset \xi_{\M}^{\ast}$ of a white vertex $w:=w_i$ with valence $v:=v_i$ corresponds to a $v/2$ cover of a sufficiently small neighborhood of $0$.
The induced vector field can be calculated:
\begin{equation}
\left(\eta\right)_{\ast}\left(\xi_{\M}\right)=\left(z^{2/v}\right)_{\ast}\left( \vf \right)=\frac{2}{v}z^{-\left(v/2-1\right)}\vf
\end{equation}
in another sufficiently small neighborhood of $0$ since the $v/2$ covering
 has constant vector field $\vf$.  Therefore, the vector field  $\xi_{\M}$ has a pole of order $v/2-1$ at $w$.\par
A neighborhood  $U_{[b]}\subset \xi_{\M}^{\ast}$ of a puncture $\infty_{[b]}$ of $\xi_{\M}^{\ast}$ that corresponds to a black vertex $b$ of $\Gamma$ with both incoming and outgoing edges
corresponds to a restricted union of half-planes and strips.
We wish to show that the holomorphic extension of $\eta$ for the  chart $\eta_{[b]}:U_{[b]}\rightarrow V^{\ast}$ to $0 \in V$ gives that the extended vector field to the puncture $\infty_{[b]}$ is a zero of order $m=r/2+1$, where $r:=r(b)$ is the number of cyclic reversals (see \S \ref{sec:admissible_graph}) of the edge orientation when traversing around $b$. We do this by an index argument. More specifically we will show that  the index of a sufficiently small curve around 0 is $m$. Consider the (piecewise smooth) Jordan curve about $\infty_{[b]}$ in $U_{[b]}$ that corresponds to the curve $\gamma_{[b]}$ in rectifying coordinates with half-circles in each upper and lower half-plane and appropriate line segments in the half-strips  with clockwise orientation so that $\infty_{[b]}$ is to the left of the curve (see Figure \ref{fig:indexfig}).
\begin{figure}[htbp]%
\large
        \includegraphics[width=.45\textwidth]{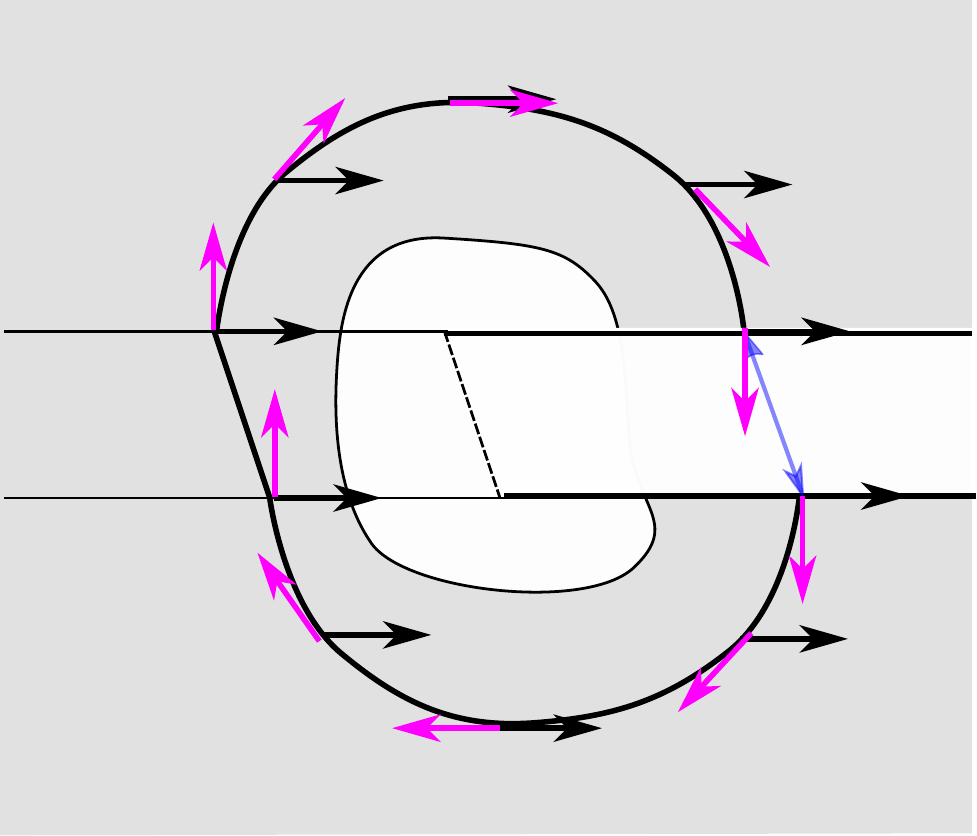}%
        \hspace{.6in}
        \includegraphics[width=.45\textwidth]{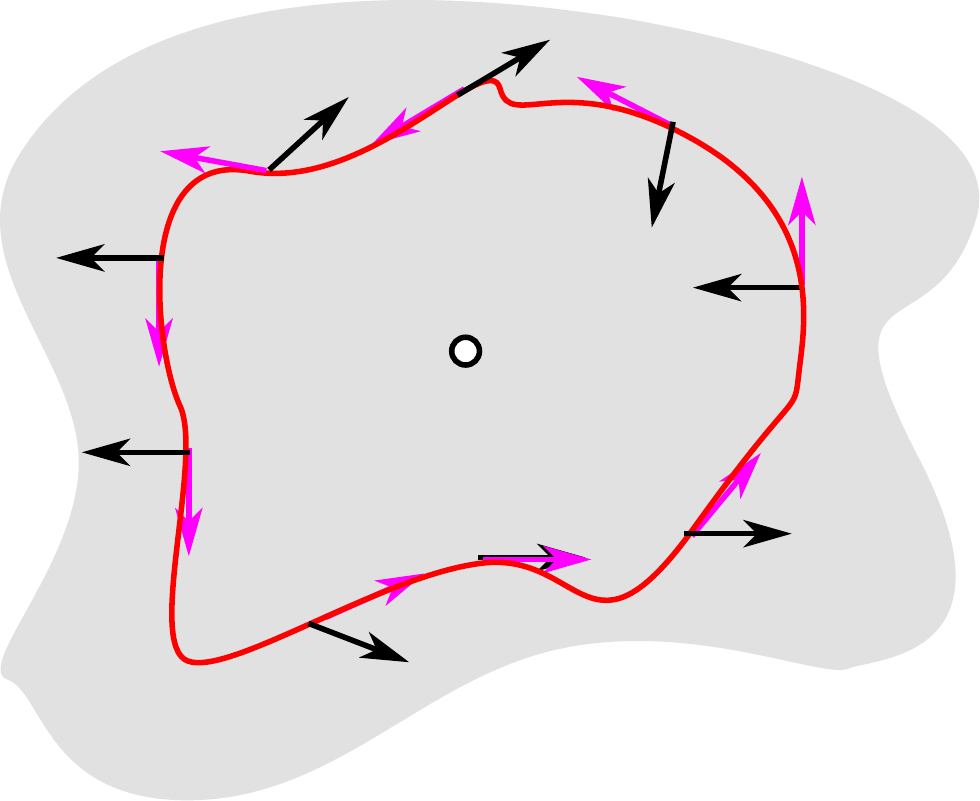}%
        \put(-470,150) {$U_{[b]}$ }
        \put(-410,28) {$\gamma_{[b]}$ }
        \put(-170,20) {$\gamma$ }
        \put(-237,88) {$\longrightarrow$ }
        \put(-230,98) {$\eta$ }
\caption{A neighborhood  $U_{[b]}\subset \xi_{\M}^{\ast}$ of a puncture $\infty_{[b]}$ of $\xi_{\M}^{\ast}$ for a black vertex $b$ of $\Gamma$ with both incoming and outgoing edges
corresponds to a restricted union of half-planes and strips (left). The vector field in these regions is $\vf$ (black arrows). The neighborhood of zero $V^{\ast}=\eta (U_{[b]})$ containing the simple closed curve $\gamma = \eta (\gamma_{[b]})$ (right). The pink arrows are the tangent vectors to $\gamma$ with a counterclockwise orientation, and the black arrows are the vectors of $\left( \eta \right)_{\ast}\left( \xi_{\M} \right)$ along $\gamma$.
}
\label{fig:indexfig}
\end{figure}
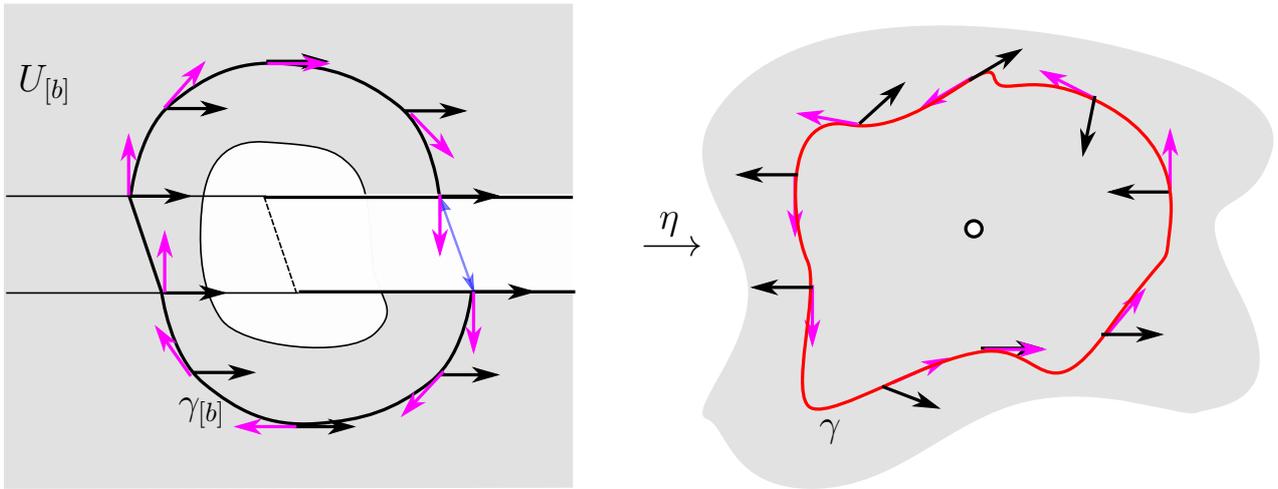
The curve $\gamma_{[b]}$ maps to another Jordan curve $\gamma$ in $V^{\ast}$ under $\eta$ since $\eta$ is univalent.  Moreover,  $\left(\eta\right)_{\ast}\left(\xi_{\M}\right)$ is never 0 along $\gamma$.  The angles between $\left(\eta\right)_{\ast}\left(\xi_{\M}\right)$ on $\gamma$ and the tangent vectors of $\gamma$ are the same as the angles between the vector field $\vf$ and the tangent vectors of $\gamma$ since $\eta$ is conformal (see Figure \ref{fig:indexfig}).

This implies that on $\gamma$, $\left(\eta\right)_{\ast}\left(\xi_{\M}\right)$ will be tangent to $\gamma$ exactly $r$ times, and alternating with pointing along the orientation of $\gamma$ and against the orientation of $\gamma$ when we travel along $\gamma$.  Along the orientation of $\gamma$, on the arcs between the tangents going along to the tangents going against, the vector field must be pointing inward, since this corresponds to what happens in rectifying coordinates.  This implies that the vector field must be rotating in the same orientation as $\gamma$, otherwise there would be a place in this arc where there is another tangent.  So with respect to the tangents of $\gamma$, $\left(\eta\right)_{\ast}\left(\xi_{\M}\right)$ has rotated $r/2$ times.  We must add one more time around, accounting for the index of $\gamma$.  This gives that $\left(\eta\right)_{\ast}\left(\xi_{\M}\right)$ has index $r/2+1$ at 0.\par
\subsection{Proof of Theorem A}
\label{finalpf}
We now show that there is a  conformal isomorphism $\Phi:\M \rightarrow \CC$ that
induces a
vector field $\xi_R= R(z) \vf= \frac{P(z)}{Q(z)}\vf$ with $\deg(P)=\deg(Q)+2$  having separatrix graph  homeomorphic to $\Gamma$.

\begin{proof}
 By the Uniformization Theorem, there exists an isomorphism $\Psi:\M \rightarrow \CC$, which is unique up to post composition by a M\"{o}bius transformation and induces a vector field  $\Psi_{\ast}\left(\xi_{\M}\right)$ defined on $\CC$.  Choose $\Psi$ such that $\Psi\left(a\right)=\infty$, where $a$ is a regular point that is not on a separatrix. Then  $\Psi_{\ast}\left(\xi_{\M}\right)$ is a meromorphic vector field on $\CC$, expressed in canonical coordinates as $g \frac{d}{dz}$. Since $g$ is meromorphic on $\CC$, it is rational; and we will show that it is of the form $\frac{P(z)}{Q(z)}\vf$ with $\deg(P)=\deg(Q)+2$.
  We know by \S \ref{assvfs} that $g$ has $q$ poles of order $m_i$ at $\Psi(w_i)$, so the degree  of $Q$ is $\sum_{i=1}^{s}2(m_i+1)=:n$. We will show that the degree of the numerator $P$ is $n+2$.\par
   Recall $r(b_i)$ is the number of edge orientation cyclic reversals at $b_i$, and let $d(b_i):=v(b_i)-r(b_i)$ be the number of non-reversals in edge orientation at $b_i$. The degree of the numerator of $R$ is the sum of equilibrium points, counting multiplicity
  \begin{equation}\label{degnum}
    \deg(P)=c+\sum_{i=1}^{p}(r(b_i)/2+1)=c+p+\frac{1}{2}\sum_{i=1}^{p} r(b_i).
  \end{equation}
 From Lemma \ref{Eulerlemma}, the embedded $\Gamma$ satisfies $v-e+f=2$.  The number of black and white vertices gives $v=q+p$. The number of edges is $e=\sum_{i=1}^{q}2(m_i+1)-h$, the number of separatrix directions, minus the number of homoclinic or heteroclinic. This further simplifies to $e=2n+2q-h$ by using $n:=\sum_{i=1}^{s}2(m_i+1)$. The number of faces is $f=c+\sharp \text{half-planes} + \sharp \text{strips}$. The number of half-planes is $\sum_{i=1}^{p}r(b_i)$, and the number of strips is $\frac{1}{2}\sum_{i=1}^{p}d(b_i)$. This gives
  \begin{align}
    v-e+f & =2 \nonumber \\
    q+p-(2n+2q-h)+ c+ \sum_{i=1}^{p}r(b_i) + \frac{1}{2}\sum_{i=1}^{r}d(b_i) & =2.
  \end{align}
  Combining this with Equation \eqref{degnum} gives
  \begin{align}\label{degnum2}
    \deg(P)=&2-q+2n+2q-h-\frac{1}{2}\sum_{i=1}^{p}r(b_i)-\frac{1}{2}\sum_{i=1}^{p}d(b_i) \nonumber \\
            =&2+2n+q-h-\frac{1}{2}\sum_{i=1}^{p}(r(b_i)+d(b_i)).
  \end{align}
  Now $\sum_{i=1}^{p}(r(b_i)+d(b_i))$ is the number of landing separatrices, which equals $\sum_{i=1}^{q}2(m_i+1)-2h$. Combining with Equation \eqref{degnum2} gives
  \begin{align}\label{degnum3}
    \deg(P)=&2+2n+q-h-(\sum_{i=1}^{q}(m_i+1)-h) \nonumber \\
            =&2+2n+q-h-(n+q-h) \nonumber \\
            =&n+2.
  \end{align}
\end{proof}
\section*{Appendix}
Here we prove Proposition \ref{lengthprop}.
\begin{proof}[Proof of Proposition \ref{lengthprop}]
  Without loss of generality, we may assume that $x_1,\dots,x_n$ are the lengths of the edges that are on the boundary of at least one of the $N$ annular regions, since all other edges have lengths which may be chosen freely. 
  Each annular region has two boundary components: one on the left of the annulus, respecting graph orientation, and one on the right. The $N$ annular regions give $N$ homogeneous linear equations in $n$ variables. Each equation in the system can be written as
  \begin{equation}\label{eq:lengthsys}
    \sum_{i \in I_{\ell}} x_i = \sum_{i \in I_{r}} x_i,
  \end{equation}
  where $I_{\ell}$ is the subset of indices corresponding to the left-hand boundary component of that annulus, and $I_{r}$ is the subset of indices corresponding to the right-hand boundary component of the same annulus. This system of equations written in the form \eqref{eq:lengthsys} has the following properties:
  \begin{enumerate}[label=(\alph*)]
    \item \label{itm:first} There is at least one $x_i$ on each side of each equation.
    \item \label{itm:second} Each $x_i$ can appear at most twice in the system.  If it does appear twice, it appears once on the left and once on the right of two different equations.
    \item \label{itm:third} If two equations have the same $x_i$ on one side, then their other sides must not have any $x_j$ in common.
  \end{enumerate}
  We explain the above properties. Property \ref{itm:first} is necessary since each annulus has at least one edge on each boundary component. Property \ref{itm:second} is due to the fact that each edge can be on the boundary of at most two annuli; and if it is on the boundary of two annuli, it is on the left-hand boundary of one annulus and the right-hand boundary of the other. Property \ref{itm:third} reflects that if two annuli share an edge on one boundary component, then their other boundary components share  no edges. Indeed, the core curves in the annuli separate these second boundary components.

  In Figure \ref{fig:sharing_edges} we show the two different configurations of two annuli sharing some edges, where we can see the above three properties into a concrete example. On the left hand side of this figure  we show the not nested case where  both annuli share the edge $x_6$. Using the notation introduced before, we have the following two equations
  \[
  \left \{
 \begin{array}{lcr}
  x_1+x_2   &= &  x_3+x_4+x_5+{\b x_6}\\
  {\b x_6} + x_7 + x_8 &= &  x_9+x_{10}
 \end{array}
 \right.
  \]

     \begin{figure}[ht]
    \centering
     \includegraphics[width=0.85\textwidth]{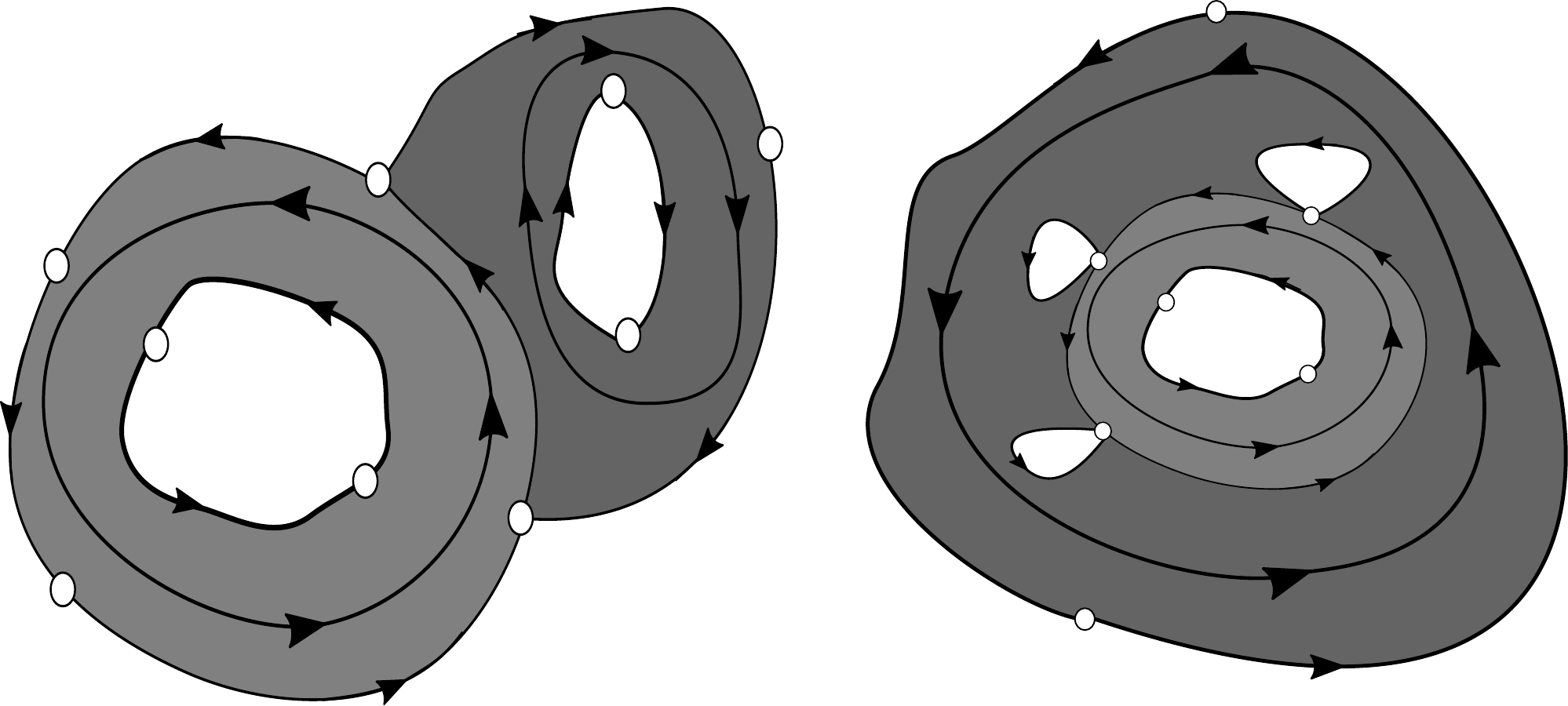}
    \put(-390,133) {$x_3$ }
     \put(-415,53) {$x_4$ }
    \put(-360,63) {$x_1$ }
        \put(-360,-1) {$x_5$ }
    \put(-300,165) {$x_7$ }
    \put(-225,50) {$x_8$ }
    \put(-245,135) {$x_9$ }
    \put(-258,110) {$x_{10}$ }
    \put(-290,120) {{\b $x_6$} }
    \put(-330,90) {$x_2$ }
    \put(-100,85) {$x_1$ }
        \put(-140,85) {{\b $x_5$ }}
                \put(-140,113) { $x_4$ }
                                \put(-140,173) { $x_{10}$ }
      \put(-140,64) { $x_6$ }
    \put(-95,105) {$x_2$ }
        \put(-95,135) {{\b $x_3$} }
                \put(-70,135) { $x_8$ }
                \put(-90,50)  {{\b $x_7$ }}
    \put(-115,10) {$x_9$ }
           \caption{Examples of two annuli sharing some edges. On the left hand side the two annuli are not nested and they share the edge $x_6$, on the right hand side  the two annuli are nested and they share the edges $x_3, x_5$ and $x_7$.}
    \label{fig:sharing_edges}
    \end{figure}
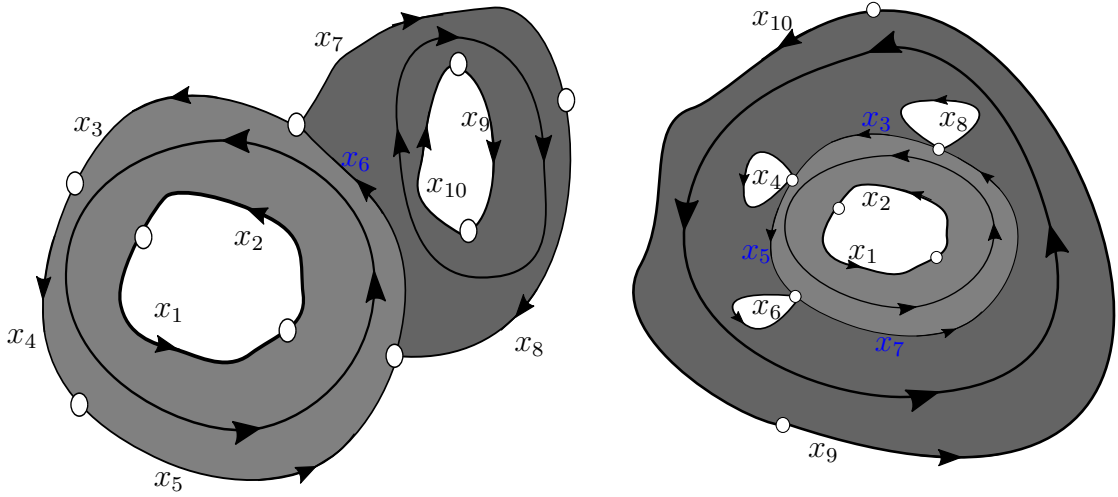

  On the right hand side of the figure we show the nested configuration where both annuli share the edges $x_3,x_5$ and $x_7$. In this concrete example the two linear equations writes as
  \[
  \left \{
 \begin{array}{lcr}
x_1+x_2 &= &   {\b x_3+x_5+ x_7}  \\
  {\b x_3}+x_4+{\b x_5}+x_6+ {\b x_7}+x_8&= & x_9 + x_{10}.
 \end{array}
 \right.
  \]

   Applying a result by Dines \cite{PositiveSolutionsLinear}
   will show that there exists a solution to this system having all positive components, henceforth called \emph{positive solutions}.

The method of Dines is an inductive method to reduce the system of equations to a system having one equation less, where the initial system admits a positive solution if and only if the new system does. A requirement to use this result is that in no step of the induction do any of the equations have all positive or all negative coefficients when written in standard form. Clearly, such an equation cannot have a positive solution. We will show that at each inductive step in our systems \eqref{eq:lengthsys} that the equations will never have all positive or all negative coefficients; equivalently, at least one term on the left and on the right have positive coefficients. We will first review the procedure of Dines and then apply it to our case. Consider a system of $N$ homogeneous linear equations of the form
\begin{equation}
  \sum_{i=1}^{n} a_{ri}x_i=0, \quad r=1,\dots,N.
\end{equation}
The procedure of Dines  first partitions the indices of the first equation according to the signs of the coefficients: let $I$ be the set of  $i$ such that $a_{1i}\geq 0$ and let $J$ be the set of $j$ such that $a_{1j}<0$. Then the first equation is rewritten as
\begin{equation}\label{eq:1steqn}
    \sum_{i\in I} a_{1i} x_i = -\sum_{j \in J}a_{1j} x_j,
  \end{equation}
  and the remaining $N-1$ equations are each written in the form
  \begin{equation}\label{eq:rtheqn}
    \sum_{j\in J} a_{rj} x_j = -\sum_{i\in I}a_{ri} x_i, \quad r=2,\dots,N,
  \end{equation}
  where it is understood that the $i$ and $j$ are from the same partition of indices as determined by the first equation. Next, the first equation is multiplied by each of the other $N-1$ equations to create a new system of $N-1$ linear equations in the new variables $x_{ij}:=x_ix_j$:
   \begin{equation}\label{eq:product}
     \sum_{i}a_{1i}x_i\cdot \sum_{j}a_{rj}x_j = \sum_{j}a_{1j}x_j\cdot \sum_{i}a_{ri}x_i.
   \end{equation}
   The procedure is repeated until there is only one remaining equation in $y_i:=x_{i_1}x_{i_2}\cdots x_{i_N}$, $i_j = 1,\dots, n$ for all $j$
  \begin{equation}
  \sum_{i=1}^{n\cdot N} b_{i}y_i=0,
\end{equation}
which has a positive solution
\begin{align}
  y_i= & -\sum_{j}b_j \\
  y_j= & \sum_{i} b_i,
\end{align}
   where again the indices $i$ and $j$ are partitioned such that $b_i \geq 0$ and $b_j<0$. Dines proves that the existence of this positive solution guarantees the existence of a positive solution for the original system, as long as at each step of the iteration, no equation has only positive or only negative coefficients (in standard form).
  \par
  Let us apply this procedure to our case. We will need to show that at each step of the procedure, all equations in the system have both positive and negative coefficients (equivalently, at least one term on each side with positive coefficient). It will be enough to show that each intermediate step satisfies the three properties above.
  Applying to the annulus problem, the equations \eqref{eq:1steqn} and \eqref{eq:rtheqn} take the form
  \begin{equation}\label{eq:A1steqn}
    \sum_{\{i\}\subseteq I}  x_i = \sum_{j\in J} x_j, \quad r=1
  \end{equation}
  and
  \begin{equation}\label{eq:Artheqn}
    \sum_{\{j\}\subseteq J}  x_j = \sum_{\{i\}\subseteq I}\pm x_i, \quad r=2,\dots,N,
  \end{equation}
  since all coefficients are $0$, $1$ or $-1$.  Note that both sides of the first equation have only positive coefficients as written.  The $r_{th}$ equation as presented has left hand side which  has only non-negative coefficient terms because of property \ref{itm:second}.  If the left is non-zero, the right hand side has at least one positive coefficient term by property \ref{itm:first}; if the left equals zero, the right hand side has at least one positive and one negative term also by property \ref{itm:first}. Since both sides of \eqref{eq:A1steqn} have  positive coefficients, the signs of the product coefficients of \eqref{eq:A1steqn} and \eqref{eq:Artheqn} is determined by the signs of the  coefficients of \eqref{eq:Artheqn}.  Hence, multiplying the left sides of \eqref{eq:A1steqn} and \eqref{eq:Artheqn} gives only positive coefficient terms or zero on the left, and as before, at least one positive coefficient term on right in the first case and at least one positive and one negative in the second case.
   \begin{align}
     \text{first equation} \quad (+)&=(+)   \\
     \text{rth equation} \quad (+)&=(+)+(+ \ or \ -) \quad \text{or}\quad 0=(+)+(-)+(+ \ or \ -)    
   \end{align}
   Therefore the equations in the resulting system have not all positive or all negative coefficients in standard form (Property \ref{itm:first}). There is, however, a detail that needs to be considered: it is \emph{a priori} possible that some of the terms on the left may be equal to some on the right in the product, and simplifying could hypothetically cause the equation to only have positive or negative coefficient terms in standard form. Let us see that this cannot happen with our system. You will get cancelling terms if the product of \eqref{eq:A1steqn} and \eqref{eq:Artheqn} has some $x_{i_0j_0}$ both on the right and on the left.  The $i_0$ will come from the left side of \eqref{eq:A1steqn} and the $j_0$ would have to be from the left side of \eqref{eq:Artheqn}. Now the indices $j$ on the left of \eqref{eq:Artheqn} are some subset of the $j$ on right of \eqref{eq:A1steqn}. The right of \eqref{eq:Artheqn} gets multiplied by all of these $j$, so you would only get the same pair $i_0j_0$ if the right side of equation \eqref{eq:Artheqn} contains some $i$ that is on left of \eqref{eq:A1steqn}. This cannot happen by property \ref{itm:third} on \eqref{eq:A1steqn} and \eqref{eq:Artheqn}. \par
  Now we need to show that the new system of equations in $x_{ij}$ satisfies the same three properties so that we can iterate this argument. First note that the coefficients of the product system are $0$, $1$, and $-1$. The new system inherits property \ref{itm:first} by the paragraph above.  We now show that it also inherits property \ref{itm:second} That is, every $x_{ij}$ can appear at most twice in the $N-1$ equations, and if it appears twice they are in different equations with opposite sign in standard form (or each positive on opposite sides). By the argument in the preceding paragraph, no $x_{i_0j_0}$ can be both on the left and right of the same equation. The variable $x_{i_0j_0}$ can not appear anywhere else on the left in the product system \eqref{eq:product}, since the $j_0$ comes from the left of \eqref{eq:Artheqn}, and this cannot appear anywhere else on the left by property \ref{itm:second} since it is  on the right of \eqref{eq:A1steqn}. The same $x_{i_0j_0}$ can appear at most once on the right since $i_0$ already appears on left in \eqref{eq:A1steqn} so can appear at most once on right of the system (and both have positive coefficients.) \par
  It now remains to show that if two equations in the product have an $x_{i_0j_0}$ once on each a respective side, that the opposite sides of the same equations share no other $x_{i_1j_1}$.  If the product gives a common $x_{i_0j_0}$, it must stem from three equations
  \begin{align}
    (r_1) && \ \dots + x_{i_0}+\dots  = & \dots + x_{j_0}+\dots \\
    (r_{m_1}) && \ (\ast) \ = &  \dots + x_{i_0}+\dots \\
    (r_{m_2}) && \ \dots + x_{j_0}+\dots = & \ (\ast \ast).
  \end{align}
  Note that both sides of $(r_1)$ are disjoint,  $(\ast)$ is disjoint from the right of $(r_1)$: $\dots + x_{j_0}+\dots$, and $(\ast \ast)$ is disjoint from the left of $(r_1)$: $\dots + x_{i_0}+\dots$. The product of $(r_1)$ with the other two gives
  \begin{align}
    \dots + x_{i_0j_0}+\dots & = (\dots + x_{j_0}+\dots)(\ast \ast)  \\
    (\ast)(\dots + x_{i_0}+\dots) & = \dots + x_{i_0j_0}+\dots,
  \end{align}
  and $(\dots + x_{j_0}+\dots)(\ast \ast)$ and $(\ast)(\dots + x_{i_0}+\dots)$ are disjoint since, even if $(\ast \ast)$ and $(\ast)$ are not disjoint (have some same index), the other index in the product will be different.
\end{proof}
\bibliographystyle{alpha}
\bibliography{Biblio_Dias_edit}
\addcontentsline{toc}{section}{References}
\end{document}